\documentclass[11pt]{article}
\usepackage{amsthm,amsmath,amssymb,amsfonts,url,booktabs,setspace,fancyhdr,bm}
\usepackage{cancel}
\usepackage{graphicx}
\usepackage{geometry}
\usepackage{enumerate}
\usepackage[shortlabels]{enumitem}
\usepackage[babel]{microtype}
\usepackage[english]{babel}
\usepackage{comment}
\usepackage{bbm}
\usepackage{csquotes}
\usepackage{mathabx}
\usepackage{subcaption}
\usepackage{float}
\usepackage{xcolor}
\usepackage{diagbox}
\usepackage{tikz}
\usetikzlibrary{positioning,arrows.meta,shapes.geometric}
\usepackage{hyperref}
\usepackage[capitalise]{cleveref}

\theoremstyle{plain}
\newtheorem{theorem}{Theorem}[section]
\newtheorem{lemma}[theorem]{Lemma}
\newtheorem{claim}[theorem]{Claim}

\newtheorem{fact}[theorem]{Fact}

\newtheorem{corollary}[theorem]{Corollary}
\newtheorem{cor}[theorem]{Corollary}
\newtheorem{conjecture}[theorem]{Conjecture}
\newtheorem{problem}[theorem]{Problem}
\newtheorem{definition}[theorem]{Definition}

\newtheorem{remark}[theorem]{Remark}
\newtheorem{construction}[theorem]{Construction}

\newlist{Case}{enumerate}{3}
\setlist[Case,1]{label={\bfseries Case \arabic*.},labelindent=1em,labelwidth=1cm,labelsep*=1em,leftmargin=!}
\setlist[Case,2]{label={\bfseries Subcase \arabic{Casei}.\arabic*.},labelindent=-1em,labelwidth=1cm,labelsep*=1em,leftmargin=!}
\setlist[Case,3]{label={\bfseries Subsubcase \arabic{Casei}.\arabic{Caseii}.\arabic*.},labelindent=-1em,labelwidth=1cm,labelsep*=1em,leftmargin=!}

\newenvironment{poc}{\begin{proof}[Proof of claim]}{\end{proof}}

\newcommand{\eps}{\varepsilon}

\newcommand{\VCdim}{\mathrm{VC}}

\newcommand{\ex}{\operatorname{ex}}

\title{Homomorphism and VC-dimension thresholds:\\ spectra and separations}
\author{
Lior Gishboliner\thanks{Department of Mathematics, University of Toronto, Canada. Email: \texttt{lior.gishboliner@utoronto.ca}. Research supported by an NSERC Discovery Grant.
}
\and
Xinqi Huang\thanks{School of Mathematical Sciences, University of Science and Technology of China, Hefei, China and Extremal Combinatorics and Probability Group (ECOPRO), Institute for Basic Science (IBS), Daejeon, South Korea.
Email: \texttt{huangxq@mail.ustc.edu.cn}. Supported by the USTC Excellent PhD Students Overseas, the Institute for Basic Science (IBS-R029-C4), the National Key Research and Development Programs of China 2023YFA1010200, the NSFC under Grants No. 12171452 and No. 12231014 and Innovation Program for Quantum Science and Technology 2021ZD0302902.
}
\and
Hong Liu\thanks{Extremal Combinatorics and Probability Group (ECOPRO), Institute for Basic Science (IBS), Daejeon, South Korea. Email: \texttt{hongliu@ibs.re.kr}. Supported by the Institute for Basic Science (IBS-R029-C4).}
}
\date{}

\begin{document}

\maketitle

\begin{abstract} Minimum-degree thresholds ask when excluding a fixed graph $H$ forces a dense graph to admit a simple global description. For each fixed chromatic number, the chromatic threshold has only three possible values. We show that this finite-spectrum phenomenon is special to chromatic threshold: already among $3$-chromatic graphs, both the homomorphism and VC-dimension thresholds have infinite spectra and are nonmonotone under taking induced subgraphs. For complete tripartite graphs with a singleton part, we prove $\delta_{\mathrm{hom}}(K_{1,s,t}) \ge \max\left\{\frac13,\frac{s}{1+s+t}\right\}$, with equality for an infinite range of $s,t$; in particular, $\delta_{\mathrm{hom}}(K_{1,s,s})=s/(2s+1)$ for every $s\ge2$. More generally, for every $r\ge3$, the value $(r-2)/(r-1)$ is an accumulation point of the homomorphism thresholds of $r$-chromatic graphs. For maximal $H$-free graphs, we determine the VC-dimension threshold of every complete tripartite graph and prove that it is positive for every nonbipartite $H$, yielding in particular the exact value for every odd cycle. We also classify the chromatic threshold under an a priori VC-dimension bound.

Together with known blowup-threshold results, our theorems reveal that $\delta_\chi,\delta_{\mathrm{hom}},\delta_{\mathrm{VC}}$, and $\delta_{\mathrm B}$ are \emph{pairwise distinct}: bounded colorability, homomorphic compressibility, neighborhood complexity, and exact blowup structure are genuinely different forms of global simplicity. The proofs develop random and grid-based obstructions to bounded homomorphic images, saturated gadgets that preserve high VC-dimension under maximal completion, and a core-orientation method for raising minimum degree while preserving $H$-freeness. 
\end{abstract}

\section{Introduction}
A minimum-degree condition can turn the local prohibition of a fixed graph
$H$ into a global structural theorem.  The thresholds studied here measure
several different ways in which a dense $H$-free graph can become simple:
it may have bounded chromatic number, factor through a bounded $H$-free
homomorphic target, have a neighborhood set system of bounded VC-dimension,
or, in the strongest exact form, be a blowup of a bounded-size graph. These
conclusions are related, but they retain different amounts of information
about the original graph and are not formally equivalent.

The chromatic threshold provides a striking benchmark.  The
\emph{chromatic threshold} $\delta_\chi(H)$ is the infimum of all $\gamma>0$
such that every $H$-free graph $G$ with
$\delta(G)\ge \gamma |V(G)|$ has chromatic number bounded by a constant
depending only on $\gamma$ and $H$.  Erd\H{o}s and Simonovits
\cite{erdos_valence_1973} initiated the study of this parameter and conjectured
that $\delta_\chi(K_3)=1/3$; this was proved by Thomassen
\cite{thomassen_chromatic_2002}.  Following a sequence of works
\cite{goddard_dense_2011,luczak_coloring_2010,nikiforov_chromatic_2010},
Allen, B\"ottcher, Griffiths, Kohayakawa and Morris
\cite{ALLEN2013261} determined $\delta_\chi(H)$ for every graph $H$.  Their
theorem shows that, if $r=\chi(H)\ge3$, then
\begin{equation}\label{eq:chromatic threshold values}
\delta_{\chi}(H) \in
\left\{
\frac{r-3}{r-2},
\frac{2r-5}{2r-3},
\frac{r-2}{r-1}
\right\}.
\end{equation}
Thus, at each fixed chromatic number, the chromatic threshold has a
three-point spectrum.  This raises a natural question: is such finite
quantization a general feature of minimum-degree structure, or is it specific
to bounded colorability?

Bounded chromatic number is equivalent to a homomorphism into a bounded-size
clique.  Requiring the target itself to remain $H$-free gives a substantially
more faithful structural description. Formally, the \emph{homomorphism threshold} $\delta_{\mathrm{hom}}(H)$ of a graph $H$ is defined as follows:
\[
\delta_{\mathrm{hom}}(H):=
\inf\Bigl\{\gamma\ge 0:\exists\ H\textup{-free }F=F(\gamma,H)\
\textup{ s.t. \(\forall\ H\)-free \(G\), if }
\delta(G)\ge \gamma |V(G)|,
\textup{ then } G\xrightarrow{\textup{hom}}F\Bigr\}.
\]
Clearly $\delta_\chi(H)\le \delta_{\mathrm{hom}}(H)$.  Thomassen
\cite{thomassen_chromatic_2002} conjectured that
$\delta_{\mathrm{hom}}(K_3)=1/3$, and this was proved by \L{}uczak
\cite{luczak_structure_2006}.  The homomorphism thresholds of cliques were
subsequently determined in
\cite{goddard_dense_2011,oberkampf_structure_2020}, and the theory has more
recently been extended to clique fans~\cite{huang_exact_2026}.  In all these
examples the chromatic and homomorphism thresholds coincide.  Odd cycles give
the first known separation: for $k\ge2$,
$\delta_{\mathrm{hom}}(C_{2k+1})\le1/(2k+1)$ by
\cite{2020COMBHomoOddCycle}, while Sankar~\cite{sankar2022homotopy} proved a
positive lower bound; in contrast,
$\delta_\chi(C_{2k+1})=0$~\cite{2007OddCycleChromatic}. However, the exact value of $\delta_{\mathrm{hom}}(C_{2k+1})$ remains unknown.

The main message of this paper is that the three-value phenomenon is special
to the chromatic threshold. The homomorphism threshold has an infinite spectrum
at every fixed chromatic number, with accumulation at the largest value in
\eqref{eq:chromatic threshold values}. Furthermore, unlike the chromatic threshold, the homomorphism threshold is not monotone under induced subgraphs. The VC-dimension threshold (defined below) likewise
has new exact values, accumulation, and non-monotonic behavior under induced
subgraphs.  We determine it for every complete tripartite graph and prove a
universal lower bound that is strictly larger than the smallest chromatic
threshold value.  We also isolate and solve the bounded-VC step in chromatic
threshold arguments.  Together with known results on blowup thresholds, these
theorems show that the four parameters encode genuinely different forms of
structure.

\subsection{Homomorphism thresholds: new values and accumulation}
Our first theorem gives a general lower bound for complete tripartite graphs
with a singleton part and determines the threshold for an infinite range of
parameters.
\begin{theorem}\label{thm: general lowbd for hom chi=3 and exact values}
For all integers $1\le s\le t$,
\[
\delta_{\mathrm{hom}}(K_{1,s,t})
\ge
\max\left\{\frac13,\frac{s}{1+s+t}\right\}.
\]
Moreover, if $(t-s+1)^2\le s$ and 
$3s\ge2+2t,$ then
\[
\delta_{\mathrm{hom}}(K_{1,s,t})=
\frac{s}{1+s+t}.
\]
\end{theorem}
The term $1/3$ is inherited from
$\delta_{\mathrm{hom}}(K_{1,s,t})
\ge\delta_\chi(K_{1,s,t})=1/3$~\cite{ALLEN2013261}; the second term is new.
On the diagonal, the theorem gives 
$\delta_{\mathrm{hom}}(K_{1,s,s})=\frac{s}{2s+1}$ for every $s\ge 2$,
already producing infinitely many exact values.  It also shows that induced
containment does not order homomorphism thresholds, e.g.~$K_{1,1,1}\subset_{\mathrm{ind}}K_{1,4,4}
\subset_{\mathrm{ind}}K_{1,4,5},$
whereas
\[
\delta_{\mathrm{hom}}(K_{1,1,1})=\frac13
<\delta_{\mathrm{hom}}(K_{1,4,4})=\frac49
>\delta_{\mathrm{hom}}(K_{1,4,5})=\frac25.
\]

\noindent
We expect the lower bound in the theorem to be sharp throughout.
\begin{conjecture}\label{conj:delta_hom K(1,s,t)}
For all integers $1\le s\le t$,
$\delta_{\mathrm{hom}}(K_{1,s,t})=
\max\left\{\frac13,\frac{s}{1+s+t}\right\}.$
\end{conjecture}
The case $K_{1,6,8}$ is a particularly interesting unresolved instance; the
relevant coloring obstruction is closely related to the
Borodin--Kostochka conjecture~\cite{borodin1977upper}.

For any threshold parameter $\delta_*$, write
\[
\Delta_*^{(r)}:=\{\delta_*(H):\chi(H)=r\}.
\]
The preceding theorem shows that $1/2$ is an accumulation point of
$\Delta_{\mathrm{hom}}^{(3)}$.  Our next result lifts this phenomenon to every
chromatic number.
\begin{theorem}\label{thm: accumulation pts for hom}
For every integer $r\ge3$, the value
$\frac{r-2}{r-1}$ is an accumulation point of
$\Delta_{\mathrm{hom}}^{(r)}$.
\end{theorem}
Thus the $r$-chromatic homomorphism spectrum is infinite for every $r\ge3$.
Moreover, its accumulation point is exactly the largest of the three
chromatic-threshold values in~\eqref{eq:chromatic threshold values}.

\subsection{VC-dimension thresholds}
VC-dimension measures the complexity of the neighborhood set system of a
graph.  Its relevance to minimum-degree problems first appeared in the work
of \L{}uczak and Thomass\'e~\cite{luczak_coloring_2010}, building on an
observation of Brandt~\cite{Brandt_Cube}: maximal triangle-free graphs above
the $1/3$ minimum-degree threshold have bounded VC-dimension, and this bounded
complexity can be converted into bounded chromatic number.  Maximality is
essential: without it, sparse parts carrying an arbitrary set system can be
added without affecting the dense structure.  This led Huang,
Liu, Rong and Xu~\cite{huang_interpolating_2025} to the following notion of \emph{VC-dimension threshold}:
\[
\delta_{\mathrm{VC}}(H):=
\inf\Bigl\{\gamma\ge 0:\exists\, C=C(\gamma,H)
\textup{ s.t. \(\forall\) maximal \(H\)-free \(G\), if }
\delta(G)\ge \gamma |V(G)|,
\textup{ then } \VCdim(G)\le C\Bigr\}.
\]

\noindent
We determine this threshold for every complete tripartite graph.
\begin{theorem}\label{thm:complete results on tripartite graphs on VC}
For all integers $1\le r\le s\le t$,
\[
\delta_{\mathrm{VC}}(K_{r,s,t})=
\begin{cases}
\frac13, & r=s=1,\\[2mm]
\frac{s}{2s+1}, & r=1,\ s\ge2,\ t\ge2s-2,\\[2mm]
\frac12, & \text{otherwise.}
\end{cases}
\]
\end{theorem}
For a singleton smallest part and $s\ge2$, this gives a sharp phase transition
at $t=2s-2$; in all remaining tripartite cases the threshold is $1/2$.
The parameter is again non-monotone under induced subgraphs, since
\[
\delta_{\mathrm{VC}}(K_{1,3,3})=\frac12
>
\delta_{\mathrm{VC}}(K_{1,3,4})=\frac37.
\]
The theorem also makes $1/2$ an accumulation point of
$\Delta_{\mathrm{VC}}^{(3)}$.

Our second VC-dimension result is a universal lower bound, which implies that the VC-dimension threshold for every nonbipartite $H$ is always positive.
\begin{theorem}\label{thm:strict lowbd for delta_VC}
If $\chi(H)=r\ge3$, then
\[
\delta_{\mathrm{VC}}(H)
\ge
\frac{(r-3)(v(H)-1)+1}{(r-2)(v(H)-1)+1}
>
\frac{r-3}{r-2}.
\]
\end{theorem}
The theorem shows that the smallest value in
the chromatic spectrum never occurs as a VC-dimension threshold.  For
$H=C_{2k+1}$, the theorem gives
$\delta_{\mathrm{VC}}(C_{2k+1})\ge1/(2k+1)$, matching the upper bound from
\cite{huang_interpolating_2025}; consequently
$\delta_{\mathrm{VC}}(C_{2k+1})=\frac1{2k+1}.$

\subsection{Separation of the four thresholds}
A blowup of a graph $F$ is obtained by replacing each vertex of $F$ by an
independent set and each edge by a complete bipartite graph. 
We write $G=F[\cdot]$ to mean that $G$ is a blowup of $F$.
A natural extension of the homomorphism threshold is the following \emph{blowup threshold}:
\[
\delta_{\mathrm{B}}(H):=
\inf\big\{\gamma\ge 0:\exists~F=F(\gamma,H)\textup{ s.t. \(\forall\) maximal \(H\)-free \(G\), if }
\delta(G)\ge \gamma |V(G)|,
\textup{ then } G=F[\cdot]\big\}.
\]

A bounded blowup representation gives both a bounded $H$-free homomorphic
target and bounded VC-dimension.  Hence
\[
\delta_\chi(H)\le\delta_{\mathrm{hom}}(H)\le\delta_{\mathrm B}(H),
\qquad
\delta_{\mathrm{VC}}(H)\le\delta_{\mathrm B}(H).
\]
The homomorphism and VC-dimension thresholds lie on different branches of this
hierarchy, so there is no formal comparison between them.  The blowup
threshold was introduced systematically in
\cite{huang_interpolating_2025}; see also the recent developments in
\cite{huang2026spectrum}.

Combining our results with known blowup thresholds separates every pair of the four parameters.
\begin{corollary}\label{thm:chi hom VC and B are distinct}
The four graph parameters
$\delta_\chi$, $\delta_{\mathrm{hom}}$, $\delta_{\mathrm{VC}}$ and
$\delta_{\mathrm B}$ are pairwise distinct.
\end{corollary}
\begin{proof}
For $H=K_{1,2,2}$,
\[
\delta_\chi(H)=\frac13,
\qquad
\delta_{\mathrm{hom}}(H)=\delta_{\mathrm{VC}}(H)=\frac25,
\qquad
\delta_{\mathrm B}(H)=\frac12,
\]
where the first equality follows from~\cite{ALLEN2013261}, the middle two
from our theorems, and the last from~\cite{huang2026spectrum}.  To separate
$\delta_{\mathrm{hom}}$ from $\delta_{\mathrm{VC}}$, take
$H=K_{1,3,3}$, for which
\[
\delta_{\mathrm{hom}}(H)=\frac37
\qquad\text{and}\qquad
\delta_{\mathrm{VC}}(H)=\frac12.\qedhere
\]
\end{proof}

Despite these separations, the VC-dimension and blowup thresholds retain a
coarse ordering by chromatic number.
\begin{cor}\label{cor:separate thresholds by chromatic number}
If $2\le\chi(H_1)<\chi(H_2)$, then
$\delta_{\mathrm{VC}}(H_1)
\le\delta_{\mathrm B}(H_1)
<\delta_{\mathrm{VC}}(H_2)
\le\delta_{\mathrm B}(H_2).$
\end{cor}
Indeed, if $\chi(H_1)=r_1$ and $\chi(H_2)=r_2$, then the Erd\H{o}s--Stone
theorem gives
$\delta_{\mathrm B}(H_1)\le(r_1- \nolinebreak 2)/(r_1- \nolinebreak 1)$, whereas
\Cref{thm:strict lowbd for delta_VC} gives
$\delta_{\mathrm{VC}}(H_2)>(r_2-3)/(r_2-2)\ge(r_1-2)/(r_1-1)$.

\subsection{Chromatic thresholds under bounded VC-dimension}
There is no direct formal comparison between $\delta_\chi(H)$ and
$\delta_{\mathrm{VC}}(H)$. However,
bounded VC-dimension helps with bounding chromatic number, as shown by \L{}uczak and Thomass\'e~\cite{luczak_coloring_2010}. To measure this influence, the following \emph{bounded-VC chromatic threshold} was introduced in~\cite{liu_beyond_2024}. For a graph $H$, let $\delta_{\chi}^{\mathrm{VC}}(H)$ be the infimum of all
$\gamma>0$ such that, for every $d\ge1$, every $H$-free graph $G$ with
$\delta(G)\ge\gamma |V(G)|$ and $\VCdim(G)\le d$ has chromatic number at
most $C(\gamma,d,H)$. In this notion, $\delta_{\chi}^{\mathrm{VC}}(K_3)=0$ whereas $\delta_{\chi}(K_3)=1/3$.

By definition, we have $\delta_{\chi}^{\mathrm{VC}}(H)\le \delta_\chi(H)$ and
\begin{equation}\label{eq:chromatic threshold via VC}
\delta_\chi(H)
\le
\max\bigl\{\delta_{\mathrm{VC}}(H),
\delta_{\chi}^{\mathrm{VC}}(H)\bigr\}.
\end{equation}
We determine the bounded-VC chromatic threshold completely.  Recall that, for a graph $H$ with $\chi(H)=r$, the decomposition family $\mathcal M(H)$ consists of all graphs
$H[B]$ arising from a partition
$V(H)=A_1\cup\cdots\cup A_{r-2}\cup B$ in which $A_1,\ldots,A_{r-2}$ are independent.
\begin{theorem}\label{thm:value for graphs of delta_chi(VC)}
Let $\chi(H)=r\ge 3$. Then
$\delta_{\chi}^{\mathrm{VC}}(H)
\in
\left\{\frac{r-3}{r-2},\frac{r-2}{r-1}\right\},$
and the larger value occurs if and only if $\mathcal M(H)$ contains no
forest.
\end{theorem}
This dichotomy parallels the decomposition-family criterion in the
classification of chromatic thresholds, but the intermediate chromatic value
disappears entirely.

The theorem supports the following general comparison.
\begin{conjecture}\label{conj:chromatic VC}
For every graph $H$,
$\delta_\chi(H)\le\delta_{\mathrm{VC}}(H).$
\end{conjecture}
\noindent
We prove the conjecture in two of the three cases in
\eqref{eq:chromatic threshold values}.
\begin{corollary}
Let $\chi(H)=r$.  If
$\delta_\chi(H)\in
\left\{\frac{r-3}{r-2},\frac{2r-5}{2r-3}\right\},$
then $\delta_\chi(H) \leq \delta_{\mathrm{VC}}(H)$.
\end{corollary}
\begin{proof}
By the theorem of Allen et al.~\cite{ALLEN2013261}, the assumption is
equivalent to $\mathcal M(H)$ containing a forest. Hence,
\Cref{thm:value for graphs of delta_chi(VC)} gives
$\delta_{\chi}^{\mathrm{VC}}(H)=(r-3)/(r-2)$, while
\Cref{thm:strict lowbd for delta_VC} gives
$\delta_{\mathrm{VC}}(H)>(r-3)/(r-2)$.  The conclusion follows from
\eqref{eq:chromatic threshold via VC}.
\end{proof}
\noindent
Thus only the top chromatic-threshold case
$\delta_\chi(H)=(r-2)/(r-1)$ remains open.

\subsection{Ideas and difficulties}

The proofs are governed by a tension that is largely absent from the ordinary
chromatic-threshold problem.  A homomorphism lower bound must survive every
map to a bounded target, while a VC-dimension lower bound must survive every
maximal $H$-free completion.  In both cases, the obstruction must also be
combined with large reservoirs that raise the minimum degree without creating
$H$.  Our lower-bound constructions therefore separate into a small
\emph{coding gadget} and a large \emph{degree-raising part}.  The structural
upper bounds proceed in the opposite direction: we first recover an almost
Tur\'an structure or finitely many neighborhood types, and then turn this
approximate information into a rigid finite description.

\paragraph{Homomorphism thresholds.}
The central difficulty in a homomorphism threshold lower bound is that large chromatic
number is not enough.  For every fixed target size, we need an $H$-free graph
whose every homomorphic image of that size already contains $H$.  For
$K_{1,s,t}$, we achieve this with a random coding gadget formed from two
families of cliques, placing one randomly positioned edge between every pair
of cliques from the two families.  In a bounded target, many cliques must have
the same ordered image.  The random edge positions then realise every possible
pair of positions between two repeated image patterns, forcing their union to
be a large clique.  Thus the gadget can be incorporated into an $H$-free
source, but becomes complete after bounded homomorphic compression. This gadget alone has very small minimum degree, so it is embedded into a
much larger system of degree reservoirs.  The attachments are balanced, so
that every vertex receives the required proportion of reservoir neighbours,
but deliberately incomplete, so that no neighborhood contains the complete
bipartite pattern needed to form $K_{1,s,t}$.

The upper bound begins with a very different, structural argument:
$K_{1,s,t}$-freeness implies that the graph has few triangles, so the removal lemma together with the Andr\'asfai--Erd\H{o}s--S\'os theorem yields an almost
bipartite partition with very small internal maximum degree.  An
$o(n)$ internal-degree bound, however, is still far from a bounded
homomorphic target.  The key new step is a quantitative Brooks-type argument:
the large degree across the partition forces any sufficiently complicated
internal configuration---either a large clique or a vertex with too many
internal neighbours---to have many repeated common neighbours on the other
side, which would extend it to $K_{1,s,t}$.  This gives complementary coloring
bounds on the two parts whose sum is $s+t$, and hence a homomorphism to the
$K_{1,s,t}$-free clique $K_{s+t}$.  The numerical hypotheses in
\Cref{thm: general lowbd for hom chi=3 and exact values} record precisely the
range in which the almost-bipartite reduction and this Brooks-type conversion
can both be made sharp.

For the accumulation theorem, randomness is replaced by a deterministic
grid.  Clusters in a common row or column are joined by matchings, while all
other pairs are joined by complements of matchings.  These two types of pairs
behave oppositely in the source and in a bounded image.  The row--column
geometry creates an incompatibility that excludes the forbidden complete
multipartite graph from the source.  Yet a bounded homomorphism must repeat
two entire layers, and the matching and 
co-matching pairs then
together witness every adjacency among the repeated images, producing a large
clique.  This ``sparse before compression, complete after compression''
mechanism gives the lower bounds approaching $(r-2)/(r-1)$.  Conversely,
slightly below that value, an almost Tur\'an partition has cross pairs so
dense that any internal star extends to the forbidden graph; the parts
therefore have bounded chromatic number.  Together, these arguments make
$(r-2)/(r-1)$ an accumulation point rather than an attained value in this
family.

\paragraph{VC-dimension thresholds.}
The lower-bound difficulty is now maximality (namely, a construction needs to be maximal $H$-free). One can easily insert a
high-VC bipartite graph into an $H$-free construction, but the missing edges
encoding its set system may be added when the graph is completed to a maximal
$H$-free graph.  Our basic device is to protect every encoded nonedge by
making it already $H$-saturated: adding that single edge immediately creates
a copy of $H$.  Hence every maximal $H$-free completion preserves the
original bipartite graph as a bi-induced subgraph, and therefore preserves
its VC-dimension.  For complete tripartite graphs, these protected encodings
can be coupled with large reservoirs whose neighborhoods are sufficiently
large to give the desired minimum degree but sufficiently incomplete to rule
out the common-neighborhood pattern required by $H$.  Again, the coding and
degree-raising tasks are separated, but maximality makes the coding problem
substantially more delicate.

The universal positive lower bound (Theorem \ref{thm:strict lowbd for delta_VC}) requires a further idea, because a
general $r$-chromatic graph has no canonical tripartite structure.  After
allowing $r-3$ independent classes of $H$ to lie in reservoirs, we select a
minimal homomorphic core among all possible nonbipartite remainders.  This
core is the finite skeleton that any copy of $H$ must reconstruct in the
coding part.  We build a graph that maps to the core but does not contain it:
one family of adjacencies encodes a high-VC bipartite graph, another encodes
its bipartite complement, and the remaining core edges synchronize the two
choices via matchings.  A copy of the core would then force the same encoded pair to be
both an edge and a nonedge.

The obstruction must also survive arbitrary completion.  If added edges
destroy all bi-induced copies of the high-VC graph, then many of them must
accumulate in one pair that was originally empty.  An application of the K\H{o}v\'ari-S\'os-Tur\'an theorem finds a large complete patch there, and the complementary encoding
propagates this patch into a large blowup of the selected core, hence into the
corresponding remainder of $H$.  Thus a maximal $H$-free completion cannot
erase the encoding (i.e., cannot lower the VC-dimension). Finally, an orientation of a sparse spanning subgraph of
the core prescribes asymmetric reservoir attachments.  It gives every vertex
a full reservoir of neighbours, while ensuring that any attempted copy of the
core using reservoir vertices misses one required adjacency.  The core is
therefore forced back into the coding gadget, where it has already been
excluded.  Minimal cores, complementary encoding, and oriented reservoirs
are the three ingredients that allow high VC-dimension, high minimum degree,
and $H$-freeness to coexist for an arbitrary nonbipartite $H$.

For $s\ge2$, the upper bound on $\delta_{\VCdim}(K_{1,s,t})$ starts from an almost-bipartite
partition with sparse internal graphs and very dense cross connections. The
cross degree is high enough that any internal $K_{1,s}$ would extend through
common neighbours on the other side to a copy of $K_{1,s,t}$; hence both
internal graphs are $K_{1,s}$-free.  When $t\ge2s-2$, every bipartition of
$K_{1,s,t}$ contains a $K_{1,s}$ entirely on one side.  Completing all cross
edges therefore cannot create the forbidden graph, so maximality forces the
cross pair to be complete already.  The graph is then the join of two
bounded-degree graphs and has bounded VC-dimension.  For $t\le2s-3$, this
bipartition property fails, and the lower-bound constructions
exploit exactly that failure.  The phase transition in
\Cref{thm:complete results on tripartite graphs on VC} is therefore explained
by a genuine structural change in the forbidden graph.  The remaining case
$s=1$ is reduced separately to the known structure of maximal triangle-free
graphs.

\paragraph{Chromatic thresholds under bounded VC-dimension.}
A partition lemma (derived from Haussler's packing lemma) divides a bounded-VC graph into finitely many classes
whose vertices have almost identical neighborhoods.  Hence every small set
inside one class has a large common neighborhood.  If a class contained a
forest $F\in\mathcal M(H)$, one could repeatedly pass to such common
neighborhoods and extend $F$ to a copy of $H$.  Every class is therefore
$F$-free and hence has bounded chromatic number, thus bounding the chromatic number of the whole graph.
The lower-bound examples are joins of a high-girth, high-chromatic-number graph with a
complete multipartite graph; their VC-dimension remains bounded because the
high-girth part is $C_4$-free.  This explains why the intermediate
chromatic-threshold value disappears under bounded VC-dimension: after
reducing to finitely many neighborhood types, the only remaining
obstruction is whether the decomposition family contains a forest.

\medskip

Taken together, the proofs isolate three kinds of information invisible to
ordinary colorability: repeated layers survive homomorphic compression,
saturated nonedges survive maximal completion, and minimal cores control how
dense reservoirs can be attached.  These mechanisms are what produce the
infinite spectra, accumulation points, and non-monotonicity established in
this paper.

\paragraph{Organization.}
\Cref{sec:prelim} collects the preliminary tools.  The homomorphism results
are proved in \Cref{sec:K1st hom,sec:hom-accumulation}.  The complete
tripartite VC-dimension thresholds and the general VC lower bound are proved
in \Cref{sec:Krst VC,sec:VC lower bound}, respectively.  The classification
of $\delta_{\chi}^{\mathrm{VC}}$ is proved in \Cref{sec:chromatic VC}, and the
paper concludes with some open problems in \Cref{sec:remark}.

\section{Preliminaries}\label{sec:prelim}
Throughout the paper we frequently use several auxiliary results, which we collect here for convenience.

\subsection{VC-dimension}
We first recall the relevant VC-dimension terminology. Let \(X\) be a finite set and let \(\mathcal F\subseteq 2^X\) be a set system.
A subset \(S\subseteq X\) is \emph{shattered} by \(\mathcal F\) if
$\{F\cap S:F\in\mathcal F\}=2^S$.
The \emph{VC-dimension} of \(\mathcal F\), denoted
$\VCdim(\mathcal F)$, is the maximum size of a subset of \(X\) shattered
by \(\mathcal F\).

For a graph \(G\), we will use the VC-dimension of its neighborhood set system
\[
    \mathcal N(G):=\{N_G(v):v\in V(G)\}\subseteq 2^{V(G)}.
\]
Thus $\VCdim(G):=\VCdim(\mathcal N(G))$.

For a bipartite graph $F$ with parts $A,B$, we slightly change the above definition of VC-dimension: We consider the set system $\{N_A(b) : b\in B\}$ and define the VC-dimension of $F$ to be the VC-dimension of this set system.\footnote{This is similar to the definition of VC-dimension for general graphs, but with two differences: first, we only consider shattered sets contained in $A$, and second, only vertices in $B$ are allowed to shatter a subset of $A$, whereas if we treated $F$ as a general graph then the empty set (as a subset of $A$) could also be accounted for by a vertex of $A$.}
Note that $F$ has VC-dimension at least $d$ if and only if there is $S \subseteq A$, such that for every $T \subseteq S$, there is a vertex $b \in B$ with $N_A(b) \cap S = T$.

For a bipartite graph $F$ with parts $A,B$, a {\em bi-induced copy} of $F$ in a graph $G$ is an injection $\varphi : V(F) \rightarrow V(G)$ such that for all $a \in A, b \in B$ it holds that $ab \in E(F)$ if and only if $\varphi(a)\varphi(b) \in E(G)$. Note that if $G$ has a bi-induced copy of $F$ then $\VCdim(G) \geq \VCdim(F)$.

We will need the following partitioning tool, which follows from Haussler's packing lemma for set systems of bounded
VC-dimension~\cite{1995PackingLemma}; see also
\cite{luczak_coloring_2010} and the formulation in
\cite{liu_beyond_2024}.

\begin{lemma}[Partition lemma {\cite{liu_beyond_2024}}]
    \label{lemma:Partition}
    Let $d$ be a positive integer and let $G$ be an $n$-vertex graph with
    VC-dimension at most $d$. For every $1\le a\le n$, there is a
    partition $V(G)=V_1\sqcup\cdots\sqcup V_m$ with
    $m\le e(d+1)(2e)^d\left(\frac{n}{a}\right)^d$
    such that $|N_G(u)\triangle N_G(v)|\le 2a$ for every $i\in[m]$ and
    every $u,v\in V_i$.
\end{lemma}

\noindent
We will also need the following
three elementary observations concerning VC-dimension.

\begin{fact}\label{fact:VC-implies-C4}
    If $G$ is $C_4$-free, then
    $\VCdim(G)\le 2$.
\end{fact}

\begin{proof}
    Assume that $\VCdim(G)\ge 3$, and let $S=\{x,y,z\}$ be a
    shattered set. There exist vertices $u,v$ such that
    $N_G(u)\cap S=\{x,y\}$ and $N_G(v)\cap S=\{x,y,z\}$.
    The four vertices $x,u,y,v$ are distinct, and both $u$ and $v$ are
    adjacent to $x$ and to $y$. Hence $xuyvx$ is a copy of $C_4$, a contradiction.
\end{proof}

\begin{fact}\label{fact:VC complete multipartite}
    If $G$ is complete multipartite then $\VCdim(G)\le 2$.
\end{fact}
\begin{proof}
    Let $S \subseteq V(G)$ be a shattered set. It is easy to see that $S$ cannot contain two elements from the same partite class of $G$. Also, if $S$ has three elements $x,y,z$ from distinct partite classes, then there is no vertex which is adjacent to $x$ but not to $y,z$. This shows that $|S| \leq 2$.
\end{proof}

Recall that the \emph{join} $G_1\vee G_2$ is obtained from the disjoint
union of $G_1$ and $G_2$ by adding all edges between their vertex sets.

\begin{fact}\label{fact:VC-of-join}
    For any two graphs $G_1$ and $G_2$,
    $\VCdim(G_1\vee G_2)
    \le
    \max\{\VCdim(G_1),\VCdim(G_2)\}+1$.
\end{fact}
\begin{proof}
    Let $S$ be shattered in $G_1\vee G_2$. If $S$ meets both
    $V(G_1)$ and $V(G_2)$, then every vertex has a neighbour in $S$, so
    the empty trace on $S$ cannot be realised. Therefore $S$ is contained
    in one side, say $S\subseteq V(G_1)$.

    Fix $x\in S$. Every subset of $S\setminus\{x\}$ is a proper subset
    of $S$, and hence cannot be realised by a vertex of $G_2$, since each
    vertex of $G_2$ is adjacent to all of $S$. Thus
    $S\setminus\{x\}$ is shattered by the neighborhood system of
    $G_1$. Consequently
    $|S|-1\le \VCdim(G_1)$, which proves the
    claim.
\end{proof}

\subsection{Classical results}

We begin with the classical bound on the Turán number of $r$-partite $r$-uniform hypergraphs. 
    Let $r\ge 2$ be an integer and let $F$ be an $r$-uniform hypergraph.
The \emph{Turán number} $\ex_r(n,F)$ of $F$ is the maximum number of hyperedges in an $n$-vertex $r$-uniform
hypergraph that does not contain $F$ as a subhypergraph.

\begin{theorem}[Multipartite Turán number~\cite{erdos1964extremal}]\label{thm:erdos-rpartite-turan}
Fix an integer $r\ge 2$ and let $F$ be an $r$-uniform hypergraph that is \emph{$r$-partite}, i.e.\ $V(F)$ admits a partition $V(F)=V_1\cup\cdots\cup V_r$ such that every edge of $F$ meets each $V_i$ in exactly one vertex.
Then there exists a constant $\delta=\delta(F)>0$ such that
$\ex_r(n,F)\le n^{\,r-\delta},$
and in particular $\ex_r(n,F)=o(n^r)$ as $n\to\infty$.
\end{theorem}

\noindent
As a direct corollary we obtain the following graph-theoretic consequence.

\begin{fact}\label{fact: Ksssss-free graph has small number}
    Fix integers $r \geq 2, s \geq 1$ and let $H\subseteq K_{s,s,\ldots,s}$ be an $r$-partite graph, where $K_{s,s,\ldots,s}$ denotes the complete $r$-partite graph with $r$ parts each of size $s$.
    Then any $H$-free graph $G$ on $n$ vertices contains at most $o(n^r)$ copies of $K_r$.
\end{fact}

\begin{proof}
Let $\mathcal{K}_r(G)$ be the $r$-uniform hypergraph on vertex set $V(G)$ whose hyperedges are the vertex-sets of copies of $K_r$ in $G$.
Thus $e(\mathcal{K}_r(G))$ equals the number of copies of $K_r$ in $G$.

Suppose for contradiction that $G$ contains $\Omega(n^r)$ copies of $K_r$. Equivalently,
$e(\mathcal{K}_r(G))=\Omega(n^r).$
Since $K^{(r)}_{s,\ldots,s}$ is $r$-partite, \Cref{thm:erdos-rpartite-turan} implies that for $n$ sufficiently large,
any $r$-uniform hypergraph on $n$ vertices with $\Omega(n^r)$ edges must contain $K^{(r)}_{s,\ldots,s}$ as a subhypergraph.
Hence $\mathcal{K}_r(G)$ contains $K^{(r)}_{s,\ldots,s}$, meaning that there exist pairwise disjoint vertex sets
$V_1,\ldots,V_r\subseteq V(G)$ with $|V_i|=s$ such that every transversal $r$-tuple
$\{v_1,\ldots,v_r\}$ with $v_i\in V_i$ forms a copy of $K_r$ in $G$.

But then in $G$ every pair of vertices from different parts $V_i,V_j$ is adjacent
(otherwise some transversal $r$-tuple would fail to span a $K_r$). Therefore $G[V_1\cup\cdots\cup V_r]$ contains the complete
$r$-partite graph $K_{s,s,\ldots,s}$, and thus also contains $H\subseteq K_{s,s,\ldots,s}$ as a subgraph.
This contradicts the assumption that $G$ is $H$-free. Hence $G$ contains only $o(n^r)$ copies of $K_r$.
\end{proof}

\noindent
The following is the well-known graph removal lemma.

\begin{theorem}[Graph removal lemma {\cite{alonShapira2004,ruzsaSzemeredi1978}}]\label{thm:graph-removal}
For every graph $H$ and every $\varepsilon>0$ there exists $\delta=\delta(H,\varepsilon)>0$ such that the following holds for all sufficiently large $n$.
If an $n$-vertex graph $G$ contains at most $\delta n^{v(H)}$ (labelled) copies of $H$, then one can delete at most $\varepsilon n^2$ edges from $G$ to obtain an $H$-free graph.
\end{theorem}

\noindent
Finally, the following is the classical Andrásfai--Erdős--Sós theorem.

\begin{theorem}[Andrásfai--Erdős--Sós {\cite{andrasfaiErdosSos1974}}]\label{thm:AES-Krfree}
Let $r\ge 3$ and let $G$ be an $n$-vertex $K_r$-free graph. If
$\delta(G)>\frac{3r-7}{3r-4}n,$
then $G$ is $(r-1)$-partite.
\end{theorem}

\subsection{A useful partition lemma}

The following lemma shows that a dense $K_{1,s_1,\dots,s_{r-1}}$-free graph can be vertex-partitioned into $r-1$ parts such that every vertex has almost all of its neighbors outside of its part.

\begin{lemma}\label{lem:AES-almost-partite}
    Let $r\ge 3$ and let $s_1,\ldots,s_{r-1}\ge 1$ be integers.
    For every $\eps>0$, there exists $n_0=n_0(s_1,\dots,s_{r-1},\eps)$ such that the following holds.  Let $G$ be an
    $n$-vertex $K_{1,s_1,\ldots,s_{r-1}}$-free graph with $n\ge n_0$ and
    $\delta(G)\ge
        \left(\frac{3r-7}{3r-4}+\eps\right)n$.
    Then there is a partition
    $V(G)=V_1\cup\cdots\cup V_{r-1}$
    such that
    $\Delta(G[V_i])\le \varepsilon n$
    for every $i\in[r-1]$.
\end{lemma}

\begin{proof}
    Put
    \[
        q:=r-1,\qquad
        \theta:=\frac{3q-4}{3q-1}.
    \]
    Thus, by assumption, we have $\delta(G) \geq (\theta + \varepsilon)n$.
    Fix a constant $0<\gamma\ll \min\{1/r,\eps\}$.

    Since $G$ is $K_{1,s_1,\ldots,s_{r-1}}$-free, \Cref{fact: Ksssss-free graph has small number} (applied with
    $s = \max\{s_1,\dots,s_{r-1}\}$)
    implies that $G$ contains $o(n^r)$ copies of $K_r$.  Hence, by
    \Cref{thm:graph-removal}, for all sufficiently large $n$ we may delete
    at most $\gamma^2n^2/2$ edges from $G$ to obtain a $K_r$-free graph
    $G'\subseteq G$.
    Call a vertex $v$ \emph{bad} if
    $d_{G'}(v)<d_G(v)-\gamma n$,
    and let $Z$ be the set of all bad vertices.  If $M$ is the number of
    deleted edges, then
    \begin{align}\label{eq:small-Z-AES}
        |Z|\gamma n
        \le \sum_{v\in V(G)}\bigl(d_G(v)-d_{G'}(v)\bigr)
        =2M\le \gamma^2n^2
        \Longrightarrow
        |Z|\le \gamma n.
    \end{align}
    Let $W:=V(G)\setminus Z$ and $G'':=G'[W]$.  For every $v\in W$,
    \[
        d_{G''}(v)
        \ge d_G(v)-\gamma n-|Z|
        \ge (\theta+\eps-2\gamma)n
        >\theta |W|.
    \]
    Hence \Cref{thm:AES-Krfree} implies that $G''$ is $q$-partite.  Fix a
    partition $W=W_1\cup\cdots\cup W_q$
    such that each $W_i$ is independent in $G''$.

    We now record some simple facts regarding the sets $W_1,\dots,W_q$.
    Since every vertex of $W$ loses at most $\gamma n$ incident edges when
    passing from $G$ to $G'$, we have $d_G(x,W_i)\le \gamma n$ for every $x\in W_i$.
    Moreover, for every non-empty $W_i$ and every $x\in W_i$,
    \[
        |W_i|\le n-d_{G''}(x)
        \le (1-\theta-\eps+2\gamma)n\le
        (1-\theta)n.
    \]
    The same upper bound is trivial when $W_i=\varnothing$.
    Define $m_i:=(1-\theta)n-|W_i|\ge 0$.
    We have
    \begin{equation}\label{eq:sum-mi}
        \sum_{i=1}^q m_i
        =q(1-\theta)n-|W|= (q(1-\theta)-1)n + |Z|
        \le \left(q(1-\theta) - 1 +\gamma\right)n,
    \end{equation}
    using the value of $\theta$ and \eqref{eq:small-Z-AES}.
    For $x\in W_i$, let
    $\overline d_i(x)
    :=|W\setminus W_i|-d_G(x,W\setminus W_i)$
    be the number of non-neighbors of $x$ (in the graph $G$) in the set $W \setminus W_i$.
    Since
    \[
        d_G(x,W\setminus W_i)
        \ge d_G(x)-d_G(x,W_i)-|Z|
        \ge (\theta+\eps-\gamma)n-|Z|
    \]
    and $|W\setminus W_i|= n - |W_i| - |Z| = \theta n+m_i-|Z|$, it follows that
    \begin{equation}\label{eq:non-neighbors in W-W_i}
    \overline d_i(x)\le m_i-(\eps-\gamma)n.
    \end{equation}

    We now extend $W_1,\ldots,W_q$ to a partition $V(G)=V_1\cup\cdots\cup V_q$ as follows: For each $z \in Z$, let $h \in [q]$ such that
    $d_G(z,W_h)$ is minimal (over all indices in $[q]$), and put $z$ into $V_h$.
    Thus, for every $i \in [q] \setminus \{h\}$ it holds that
    $
    d_G(z,W_h) \le d_G(z,W_i).
    $

    Suppose for a contradiction that $\Delta(G[V_h])> \varepsilon n$ for some
    $h\in[q]$, and choose $a \in V_h$ with $d_G(a,V_h)> \varepsilon n$.
    First, note that
    $a \in V_h \setminus W_h \subseteq Z$. Indeed, this is because for every
    $v \in W_h$, it holds that
    $d_G(v,V_h) \leq (d_G(v) - d_{G'}(v)) + |Z| \leq 2\gamma n \leq \varepsilon n$,
    using that $G''[W_i]$ is independent, that $|Z| \leq \gamma n$, and that
    $\gamma \ll \varepsilon$.

    Thus, $a \in Z$. Now, for each $i\in[q]$, put $B_i:=N_G(a)\cap W_i$ and $b_i:=|B_i|$. The fact that $a \in V_h$ (i.e., the definition of the partition $V_1,\dots,V_q$) means that
    $b_i = d_G(a,W_i) \geq d_G(a,W_h)$. Hence,
    \begin{equation}\label{eq:Bi-linear}
        b_i\ge
        d_G(a,W_h) \geq
        d_G(a,V_h) - |Z|
        \ge
        \varepsilon n-|Z|\ge \frac{\varepsilon n}{2}
        \qquad\text{for every }i\in[q],
    \end{equation}
    using \eqref{eq:small-Z-AES}.
    Also,
    \begin{equation}\label{eq:sum-bi-lower}
        \sum_{i=1}^q b_i=d_G(a,W)
        \ge \delta(G)-|Z|
        \ge \left(\theta+\eps-\gamma\right)n \geq \theta n.
    \end{equation}

    Let $F:=G[B_1,\ldots,B_q]$ be the $q$-partite graph consisting of the
    edges of $G$ between the sets $B_1,\dots,B_q$.  The graph $F$ is
    $K_{s_1,\ldots,s_q}$-free, since such copy together with $a$ would
    give a copy of $K_{1,s_1,\ldots,s_q}$ in $G$, contradicting the assumption of the lemma.
    Therefore, by \Cref{fact: Ksssss-free graph has small number},
    $F$ has only $o(n^q)$ copies of $K_q$.  By
    \Cref{thm:graph-removal}, we may delete $o(n^2)$ edges from $F$ to
    obtain a $K_q$-free graph $F_0$.

    For $1\le i<j\le q$, let $\overline e_{ij}$ be the number of non-edges
    of $F_0$ between $B_i$ and $B_j$, and put $x_{ij}:=\frac{\overline e_{ij}}{b_ib_j}$.
    We use the symmetric notation
    $\overline e_{ji}:=\overline e_{ij}$ and $x_{ji}:=x_{ij}$.
    Every transversal $q$-tuple in $B_1\times\cdots\times B_q$ contains
    a non-edge of $F_0$ (because $F_0$ is $K_q$-free). Hence the union bound gives
    $\sum_{1\le i<j\le q}x_{ij}\ge 1$.

For $1 \leq i < j \leq q$, let
$\overline e^{\,F}_{ij}$ denote the number of
non-edges between $B_i$ and $B_j$ in $F$.
Note that
$
\overline e^{\,F}_{ij} \geq \overline e_{ij} - o(n^2),
$
by the choice of $F_0$.

Now fix any $i \in [q]$.
Note that by \eqref{eq:non-neighbors in W-W_i}, every $x \in B_i$ has at most
$m_i-(\eps-\gamma)n$ non-neighbors in
$\bigcup_{j: j\ne i}B_j$, because
$\bigcup_{j: j\ne i}B_j\subseteq W\setminus W_i$. Summing over $x\in B_i$, we obtain
$$
\sum_{j : j\ne i} \overline e_{ij} \leq
\sum_{j : j\ne i}\overline e^{\,F}_{ij} + o(n^2)
\le b_i\bigl(m_i-(\eps-\gamma)n\bigr) + o(n^2) \leq
b_im_i,
$$
where the last inequality uses that $b_i = \Omega(n)$, by \eqref{eq:Bi-linear}.
We now obtain
\begin{equation}\label{eq:weighted-missing}
    \sum_{j : \; j\ne i}b_jx_{ij}
    =\frac{1}{b_i}\sum_{j\ne i}\overline e_{ij}
    \le m_i.
\end{equation}
    Summing \eqref{eq:weighted-missing} over $i \in [q]$ , we get
    $$
    \sum_{1\le i<j\le q}(b_i+b_j)x_{ij} =
    \sum_{i=1}^q \sum_{j : \; j\ne i}b_jx_{ij}
    \le
    \sum_{i=1}^q m_i
    \leq
    \left(q(1-\theta) - 1 +\gamma\right)n,
    $$
    where the last inequality uses
    \eqref{eq:sum-mi}.
    Together with $\sum_{1\le i<j\le q}x_{ij}\ge 1$, this implies that there exist
    $i<j$ with
    $b_i+b_j\le
    \left(q(1-\theta) - 1 +\gamma\right)n$.
    Since $b_k\le |W_k|\le (1-\theta)n$ for every $k \in [q]$, we conclude that
    \[
        \sum_{k=1}^q b_k
        \le
        \left((2q-2)(1-\theta) - 1+\gamma\right)n.
    \]
    However,
    this contradicts \eqref{eq:sum-bi-lower} (with some room to spare), because
    $(2q-2)(1-\theta) - 1 < \theta$ for
    $\theta = \frac{3q-4}{3q-1}$.
    We conclude that $\Delta(G[V_i])\le \varepsilon n$ for every $i\in[q]$, as required.
\end{proof}

\section{New values of homomorphism threshold}\label{sec:K1st hom}
The proof of \Cref{thm: general lowbd for hom chi=3 and exact values} reduces to the following two statements.

\begin{theorem}\label{thm: general lowbd for hom K_1st}
    For all integers $1\le s\le t$,
        $\delta_{\mathrm{hom}}(K_{1,s,t})\ge \max\Big\{\frac{1}{3}, \frac{s}{1+s+t}\Big\}.$
\end{theorem}

\begin{theorem}\label{thm: partial uppbd for hom K_1st}
    For all integers $1\le s\le t$ with
    $(t-s+1)^2 \leq s$ and $3s\ge 2+2t$, $\delta_{\mathrm{hom}}(K_{1,s,t})\le \frac{s}{1+s+t}.$
\end{theorem}

\subsection{Proof of \Cref{thm: general lowbd for hom K_1st}}
Our goal is to construct, for every fixed constant $C_0>0$ and every $\varepsilon>0$, a graph $G$ with
$\delta(G)\ge \Big(\frac{s}{1+s+t}-\varepsilon\Big)|V(G)|,$
such that $G$ is $K_{1,s,t}$-free but every homomorphism $\phi:G\to\Gamma$ with $|V(\Gamma)|\le C_0$ forces $\Gamma$ to contain a copy of $K_{1,s,t}$.
The construction uses randomness: we first build a suitable auxiliary gadget by a probabilistic argument, and then incorporate it into $G$.
We begin with the key gadget lemma.

\begin{lemma}\label{lem: random construction cliques}
    For all integers $k,\ell,C \ge 1$, there exists an integer $m_0=m_0(k,\ell,C)$ such that for every integer $m\ge m_0$ there is a graph $G$ with the following properties.
    \begin{enumerate}[(i)]
        \item $V(G)=V_k\cup V_\ell$, where $G[V_k]$ is the disjoint union of $m$ copies of $K_k$ and $G[V_\ell]$ is the disjoint union of $m$ copies of $K_\ell$.
        \item For every $K_k$-copy $A\subseteq V_k$ and every $K_{\ell}$-copy $B\subseteq V_{\ell}$, there is exactly one edge between \nolinebreak $A$ \nolinebreak and \nolinebreak $B$.
        \item Every homomorphic image of $G$ on at most $C$ vertices contains a copy of $K_{k+\ell}$.
    \end{enumerate}
    We denote such a graph by $H_{k,\ell,C}^{m}$.
\end{lemma}

\begin{proof}
    We start with $m$ disjoint $K_k$-copies $A_1,\dots,A_m$ and $m$ disjoint $K_{\ell}$-copies $B_1,\dots,B_m$. Put $V_k := \bigcup_{i=1}^m A_i$ and $V_{\ell} := \bigcup_{i=1}^m B_i$ (the sets $V_k,V_{\ell}$ are disjoint).

    We now define the edges between $V_k$ and $V_{\ell}$. Write
    $A_i = \{ a_i^1,\dots,a_i^k \}$ and
    $B_i = \{b_i^1,\dots,b_i^\ell\}$.
    For each $(i,j) \in [m]^2$, choose a pair
    $c(i,j) = (p,q) \in [k] \times [\ell]$ uniformly at random and independently, and add the edge $a_i^p b_j^q$ (this is the unique edge between $A_i$ and $B_j$, as required by Property (ii)).

    Set $\beta:=\frac{1}{C^{k+\ell}}$.
    We have the following claim.

    \begin{claim}\label{claim: random coloring holds}
        With positive probability, for every pair of subsets $I,J\subseteq[m]$ with
        $|I|,|J| \geq \beta m$, and for every $(p,q) \in [k]\times[\ell]$, there is
        $(i,j) \in I \times J$ with
        $c(i,j) = (p,q)$.
    \end{claim}

    Assuming \Cref{claim: random coloring holds}, we verify Property (iii). Let $\Gamma$ be a graph on at most $C$ vertices, and let $\phi : G \rightarrow \Gamma$ be a homomorphism.
    Note that for each $i \in [m]$, there are at most $|\Gamma|^k \leq C^k$ options for the images
    $(\phi(a_i^1),\dots,\phi(a_i^k))$ of the set $A_i$. Hence, by the pigeonhole principle, there are vertices $u_1,\dots,u_k \in V(\Gamma)$ and a set $I \subseteq [m]$ with
    $|I| \geq \frac{m}{C^k}$, such that
    $(\phi(a_i^1),\dots,\phi(a_i^k)) = (u_1,\dots,u_k)$ for all $i \in I$.
    By applying the same argument to the sets $B_1,\dots,B_m$, we obtain $\ell$ vertices $v_1,\dots,v_{\ell} \in V(\Gamma)$ and a set $J \subseteq [m]$ with $|J| \geq \frac{m}{C^{\ell}}$, such that
    $(\phi(b_j^1),\dots,\phi(b_j^\ell)) = (v_1,\dots,v_{\ell})$ for all $j \in J$.
    Note that $\{u_1,\dots,u_k\}$ and $\{v_1,\dots,v_{\ell}\}$ are cliques in $\Gamma$ (since $\phi$ is a homomorphism).

    We claim $\{u_1,\dots,u_k,v_1,\dots,v_{\ell}\}$ is a clique in $\Gamma$; this would establish Property (iii). It suffices to show that for all
    $(p,q) \in [k] \times [\ell]$ it holds that $u_pv_q \in E(\Gamma)$.
    Note that $|I|,|J| \geq \beta m$.
    Hence, by \Cref{claim: random coloring holds}, there exist $(i,j) \in I \times J$ such that $c(i,j) = (p,q)$. By the definition of $G$, this means that
    $a_i^pb_j^q \in E(G)$, which in turn gives
    $u_pv_q = \phi(a_i^p)\phi(b_j^q) \in E(\Gamma)$, as required.

    It remains to prove \Cref{claim: random coloring holds}.

    \begin{poc}
        Fix $I,J\subseteq[m]$ with $|I|,|J| \geq \beta m$ and fix $(p,q)\in[k]\times[\ell]$. The probability that $c(i,j) \neq (p,q)$ for all $(i,j) \in I \times J$ equals
        \[
            \left(1-\frac1{k\ell}\right)^{|I||J|}
            \leq
            \left(1-\frac1{k\ell}\right)^{\beta^2 m^2} \leq
            e^{-\frac{\beta^2}{k\ell}m^2}.
        \]
        By the union bound over all at most $4^m$ choices for $I,J$ and all $k\ell$ choices for $(p,q)$, we see that the probability that the claim fails is at most
        $$
        4^m \cdot k\ell \cdot e^{-\frac{\beta^2}{k\ell}m^2},
        $$
        which is less than $1$ for all sufficiently large $m$. This completes the proof of the claim.
    \end{poc}
    \noindent
    With Claim \ref{claim: random coloring holds}, the proof of Lemma \ref{lem: random construction cliques} is complete.
\end{proof}

\begin{proof}[Proof of \Cref{thm: general lowbd for hom K_1st}]
    Since $\delta_{\mathrm{hom}}(K_{1,s,t})\ge \delta_{\chi}(K_{1,s,t}) = \frac{1}{3}$,
    it remains to show $\delta_{\mathrm{hom}}(K_{1,s,t})\ge \frac{s}{1+s+t}$.
    When $s = 1$, we have
    $\delta_{\mathrm{hom}}(K_{1,s,t})\ge
    \frac{1}{3}\ge \frac{s}{1+s+t}$, so it remains to prove \Cref{thm: general lowbd for hom K_1st} for $s\ge 2$.

    It suffices to show that, for every integer $C_0\ge1$ and every
    $\varepsilon>0$, there are arbitrarily large graphs $G$ such that
    \begin{itemize}
        \item $G$ is $K_{1,s,t}$-free and
        $\delta(G)\ge (\frac{s}{1+s+t}-\eps)|V(G)|$;
        \item every homomorphism $\phi:G\to\Gamma$ with
        $|V(\Gamma)|\le C_0$ forces $\Gamma$ to contain $K_{1,s,t}$.
    \end{itemize}

    We now construct such a graph.
    Choose an integer $m\ge m_0(t+1,s,C_0)$ from \Cref{lem: random construction cliques}, and let
    $H := H_{t+1,s,C_0}^m$ be the graph given by that lemma. Thus, $V(H) = V_{t+1} \cup V_s$;
    $V_{t+1}$ is the disjoint union of $(t+1)$-cliques $A_1,\dots,A_m$; $V_s$ is the disjoint union of $s$-cliques $B_1,\dots,B_m$; and there is exactly one edge between $A_i$ and $B_j$ for all $i,j \in [m]$. Next, fix
    $N \gg m,1/\varepsilon$, and add disjoint sets $X_1,\dots,X_{t+1},Y$ with $|X_i| = N$ and $|Y| = sN$.
    Put $X := \bigcup_{i=1}^{t+1} X_i$.
    We add the following edges:
    \begin{enumerate}[(1)]
        \item
        $(X,Y)$ and $(V_s,Y)$ are complete bipartite graphs.
        \item For each $j\in[m]$, write
        $V(A_j)=\{a_j^{p}: p\in[t+1]\}$. For $p\in[t+1]$ and $i\in[t+1]$, we join $a_j^{p}$ to all vertices of $X_i$ if $i-p\in\{0,1,\dots,s-1\}$ modulo $t+1$, and join $a_j^{p}$ to no vertex of $X_i$ otherwise.
    \end{enumerate}
    Let $G$ be the resulting graph.
    A routine degree computation shows that
    $$
    \delta(G)\ge sN \geq
    \left(\frac{s}{1+s+t}-\varepsilon\right)|V(G)|,
    $$
    where the last inequality holds if $N$ is large enough as a function of $m,1/\varepsilon$.
    By \Cref{lem: random construction cliques}, every homomorphic image of $G$ on at most $C_0$ vertices contains a copy of $K_{s+(t+1)}=K_{1+s+t}$, and hence also a copy of $K_{1,s,t}$.
    Thus it remains to check that $G$ is $K_{1,s,t}$-free. Equivalently, we prove the following.
    \begin{claim}
        For every $v\in V(G)$, we have $G[N(v)]$ is $K_{s,t}$-free.
    \end{claim}

    \begin{poc}
        We consider the following cases:
        \begin{itemize}
            \item $v \in Y$. Then $N(v)$ consists of isolated vertices and disjoint copies of $K_s$, and is hence $K_{s,t}$-free.
            \item $v \in X$.
            Each vertex in $X$ has exactly $s$ neighbors in every clique $A_j$. Hence $N(v)$ consists of isolated vertices and disjoint copies of $K_s$, and is hence $K_{s,t}$-free.
            \item $v \in B_i$ for some $i \in [m]$. Note that the vertices in $N(v) \cap V_{t+1}$ are isolated in $N(v)$, because $v$ has at most one neighbor in each $A_j$. Hence, $N(v)$ consists of isolated vertices, the $(s-1)$-clique $B_i \setminus \{v\}$, and the complete bipartite graph between $Y$ and $B_i \setminus \{v\}$. It is now easy to see that $N(v)$ is $K_{s,t}$-free (as $s \leq t$).
            \item $v \in A_i$ for some $i \in [m]$. Without loss of generality, $i=1$ and
            $v = a_1^1$.
            Suppose by contradiction that $N(v)$ contains a copy $K$ of $K_{s,t}$.
            We have
            $N(v) \subseteq (A_1 \setminus \{a_1^1\}) \cup X \cup V_s$.
            The vertices in $N(v) \cap V_s$ are isolated in $N(v)$, and hence cannot participate in $K$.
            Also, we must have $V(K) \cap X \neq \emptyset$, because
            $|A_1 \setminus \{a_1^1\}| = t < s+t$. Now, the vertices of $K$ in $X$ must be in the same part of the bipartition of $K$, since $X$ is independent; the other part of $K$ is contained in $A_1$. Taking any $u \in V(K) \cap X$, we see that $u$ must have at least $s+1$ neighbors in $A_1$ (i.e., the vertex $v$ as well as $s$ vertices of $K$). But by (2), every vertex in $X$ is adjacent to only $s$ vertices of $A_1$, a contradiction.       This completes the proof.\qedhere
        \end{itemize}
    \end{poc}
    
This completes the proof of \Cref{thm: general lowbd for hom K_1st}.
\end{proof}

\begin{figure}[t]
\centering
\resizebox{\textwidth}{!}{%
\begin{tikzpicture}[
  font=\small,
  line cap=round, line join=round,
  small/.style={draw, circle, thick, inner sep=1.2pt},
  box/.style={draw, thick, rounded corners},
  edge/.style={semithick, shorten <=1pt, shorten >=1pt},
  heavy/.style={very thick, shorten <=1pt, shorten >=1pt},
  gadgetR/.style={dashed, semithick, shorten <=1pt, shorten >=1pt, red!70!black},
  gadgetB/.style={dashed, semithick, shorten <=1pt, shorten >=1pt, blue!70!black},
  gadgetG/.style={dashed, semithick, shorten <=1pt, shorten >=1pt, green!55!black},
  faint/.style={text=gray!60}
]

\draw[box] (-4,2) rectangle (-1,4);
\node[anchor=east] at (-4.15,3.6) {$X=\bigsqcup_{i=1}^{3}X_i$};

\node[box, minimum width=0.95cm, minimum height=0.85cm] (X1) at (-3.4,3.1) {$X_1$};
\node[box, minimum width=0.95cm, minimum height=0.85cm] (X2) at (-2.5,3.1) {$X_2$};
\node[box, minimum width=0.95cm, minimum height=0.85cm] (X3) at (-1.6,3.1) {$X_3$};

\node at (-2.5,2.3) {\scriptsize $|X_i|=N$};

\draw[box] (-4,-0.5) rectangle (-1,1);

\node at (-2.5,0.4) {$Y$};
\node at (-2.5,-0.1) {\scriptsize $|Y|=2 N$};

\draw[heavy] (-2.5,2) -- (-2.5,1);
\draw[heavy] (-1,0.25) -- (1.85,0.25);

\draw[box] (1.85,1.75) rectangle (9.15,4);
\node[anchor=west] at (9.30,3.6) {$V_{t+1}$ \quad ($t+1=3$)};

\node[small] (a11) at (3.00,3.18) {$a_1^1$};
\node[small] (a12) at (3.90,3.18) {$a_1^2$};
\node[small] (a13) at (3.45,2.38) {$a_1^3$};
\draw[edge] (a11)--(a12);
\draw[edge] (a12)--(a13);
\draw[edge] (a13)--(a11);

\node[small] (a21) at (5.30,3.18) {$a_2^1$};
\node[small] (a22) at (6.20,3.18) {$a_2^2$};
\node[small] (a23) at (5.75,2.38) {$a_2^3$};
\draw[edge] (a21)--(a22);
\draw[edge] (a22)--(a23);
\draw[edge] (a23)--(a21);

\node[small] (a31) at (7.60,3.18) {$a_3^1$};
\node[small] (a32) at (8.50,3.18) {$a_3^2$};
\node[small] (a33) at (8.05,2.38) {$a_3^3$};
\draw[edge] (a31)--(a32);
\draw[edge] (a32)--(a33);
\draw[edge] (a33)--(a31);

\node[faint] at (7.05,3.55) {\large$\dots$};

\draw[box] (1.85,-0.8) rectangle (9.15,1);
\node[anchor=west] at (9.30,0.6) {$V_s$ \quad ($s=2$)};

\node[small] (b11) at (3.00,0.15) {$b_1^1$};
\node[small] (b12) at (3.90,0.15) {$b_1^2$};
\draw[edge] (b11)--(b12);

\node[small] (b21) at (5.30,0.15) {$b_2^1$};
\node[small] (b22) at (6.20,0.15) {$b_2^2$};
\draw[edge] (b21)--(b22);

\node[small] (b31) at (7.60,0.15) {$b_3^1$};
\node[small] (b32) at (8.50,0.15) {$b_3^2$};
\draw[edge] (b31)--(b32);

\node[faint] at (7.05,-0.35) {\large$\dots$};

\draw[gadgetR] (a11.south)      to[out=-90,in=90]  (b11.north);
\draw[gadgetB] (a12.south)      to[out=-80,in=115] (b21.north west);
\draw[gadgetG] (a13.south east) to[out=-60,in=150] (b32.north);

\draw[gadgetR] (a23.south west) to[out=-115,in=70] (b12.north);
\draw[gadgetB] (a22.south)      to[out=-90,in=90]  (b21.north);
\draw[gadgetG] (a21.south east) to[out=-70,in=120] (b31.north west);

\draw[gadgetR] (a32.south west) to[out=-120,in=45] (b11.north east);
\draw[gadgetB] (a31.south west) to[out=-105,in=65] (b22.north);
\draw[gadgetG] (a33.south)      to[out=-90,in=90]  (b32.north);

\coordinate (x1a) at (-3.58,3.43);
\coordinate (x1b) at (-3.22,3.43);

\coordinate (x2a) at (-2.68,2.77);
\coordinate (x2b) at (-2.32,2.77);

\coordinate (x3a) at (-1.23,3.24);
\coordinate (x3b) at (-1.23,2.94);

\draw[edge] (a11.west) to[out=188,in=20]   (x1a);
\draw[edge] (a11.west) to[out=212,in=-8]   (x2a);

\draw[edge] (a12.west) to[out=200,in=-20]  (x2b);
\draw[edge] (a12.west) to[out=222,in=8]    (x3a);

\draw[edge] (a13.west) to[out=168,in=-10]  (x3b);
\draw[edge] (a13.west) to[out=150,in=35]   (x1b);

\end{tikzpicture}%
}
\caption{The construction in the proof of \Cref{thm: general lowbd for hom K_1st} in the case $s=t=2$.}
\end{figure}
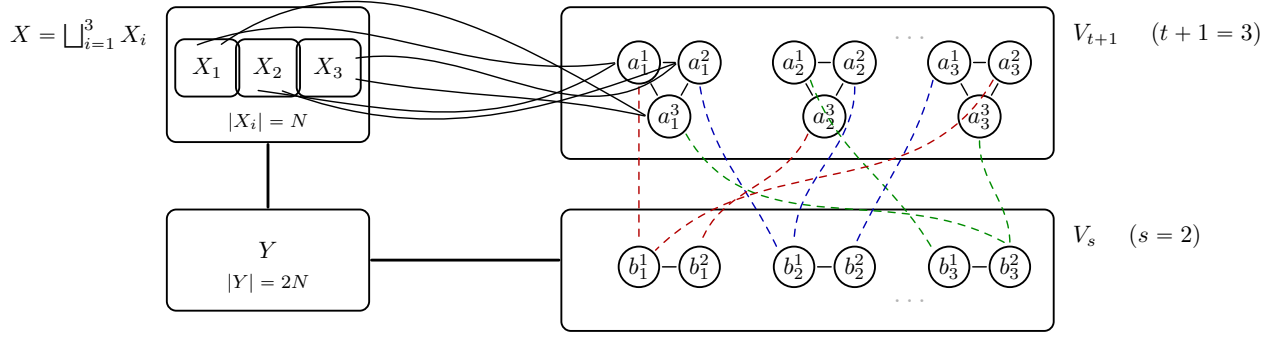

\subsection{Proof of \Cref{thm: partial uppbd for hom K_1st}}
We begin with the classical Brooks theorem.

\begin{theorem}[Brooks~\cite{brooks1941colouring}]\label{thm:brooks}
Let $G$ be a connected graph.
Then $\chi(G)\le \Delta(G),$
unless $G$ is a complete graph $K_{\Delta+1}$ or an odd cycle, in which case $\chi(G)=\Delta+1$.
\end{theorem}

The following lemma plays a key role in the proof of \Cref{thm: partial uppbd for hom K_1st}.
\begin{lemma}\label{lem:Brooks application}
    Let $1 \leq s \leq t$ and $r > s$ be integers with
    $(r-s)^2 \leq s$. Let $\varepsilon > 0$,
    and let $G$ be a $K_{1,s,t}$-free graph with a vertex partition
    $V(G) = X \cup Y$, such that
    $d_Y(x) \geq (\frac{s}{r} + \varepsilon) |Y|$ for all $x \in X$. Suppose also that
    $|Y| \gg s,t,r,1/\varepsilon$.
    Then $\chi(G[X]) \leq r-1$.
\end{lemma}
\begin{proof}
    We will apply Brooks' theorem. First we show that $G[X]$ is $K_r$-free. Indeed, suppose by contradiction that
    $\{x_1,\dots,x_r\} \subseteq X$ is a clique. We have $\sum_{i=1}^r d_Y(x_i) \geq (s+r\varepsilon)|Y|$.
    Let $Y'$ be the set of $y \in Y$ which have at least $s+1$ neighbors in
    $\{x_1,\dots,x_r\}$. Then
    $$
    \sum_{i=1}^r d_Y(x_i) \leq
    r|Y'| + s(|Y|-|Y'|) = s|Y| + (r-s) |Y'|.
    $$
    Combining this with the lower bound on $\sum_{i=1}^r d_Y(x_i)$, we get that
    $|Y'| \geq \frac{r\varepsilon}{r-s}|Y|$.
    By pigeonholing over the choice of $s+1$ neighbors of a vertex $y \in Y'$ in $\{x_1,\dots,x_r\}$,
    we see that there are
    $S \subseteq \{x_1,\dots,x_r\}$ of size $s+1$ and
    $T \subseteq Y$ of size $t$ such that $(S,T)$ is a complete bipartite graph. Since $S$ is a clique, we get a copy of $K_{1,s,t}$, a contradiction.

    Next, we show that $\Delta(G[X]) \leq r-1$. Indeed, suppose by contradiction that there are $x,x_1,\dots,x_r \in X$ with $xx_i \in E(G)$ for all $i \in [r]$. Since $G$ is $K_{1,s,t}$-free, all but $O_{s,t,r}(1)$ vertices in $N_Y(x)$ have at most $s-1$ neighbors in $\{x_1,\dots,x_r\}$; otherwise, by pigeonholing, we could find $t$ vertices in $Y$ with the same $s$ neighbors in
    $\{x_1,\dots,x_r\}$, giving a $K_{1,s,t}$.
    We conclude that
    \begin{align*}
    \sum_{i=1}^r d_Y(x_i) &\leq
    (s-1) d_Y(x) + r (|Y| - d_Y(x)) + O_{s,t,r}(1)
    \\ &=
    r|Y| - (r-s+1) d_Y(x) + O_{s,t,r}(1) <
    \left( r - \frac{s(r-s+1)}{r} \right)|Y|,
    \end{align*}
    using that
    $d_Y(x) \geq (\frac{s}{r} + \varepsilon) |Y|$.
    On the other hand,
    $\sum_{i=1}^r d_Y(x_i) > s|Y|$. Combining these two inequalities, we get that
    $r - \frac{s(r-s+1)}{r} > s$, which rearranges to $(r-s)^2 > s$, contradicting the assumption of the lemma.

    If $r=2$ or $r \geq 4$ then
    $\omega(G[X]),\Delta(G[X]) \leq r-1$ immediately gives $\chi(G[X]) \leq r-1$, by Theorem \ref{thm:brooks}.
    If $r = 3$, then we need to verify that $G[X]$ has no odd cycles. We show that in fact, $\Delta(G[X]) \leq 1$. Note that the condition $(3-s)^2 = (r-s)^2 \leq s$ implies that $s = 2$. Hence, $d_Y(x) \geq (\frac{2}{3} + \varepsilon)|Y|$ for all $x \in X$. Hence, any copy of $K_{1,2}$ inside $X$ gives a copy of $K_{1,2,t}$, since the three vertices of the $K_{1,2}$-copy have at least $\varepsilon |Y|$ common neighbors in $Y$. This contradicts $K_{1,s,t}$-freeness.
\end{proof}

\begin{proof}[Proof of \Cref{thm: partial uppbd for hom K_1st}]
    For each pair of integers $s$ and $t$ satisfying the conditions in \Cref{thm: partial uppbd for hom K_1st}, let $G$ be a $K_{1,s,t}$-free graph with $n$ vertices and $\delta(G)\ge \left( \frac{s}{1+s+t}+\eps \right) n,$
    where $0< \eps < \eps_0(s,t)$ is sufficiently small and $n\ge n_0(\varepsilon)$ is sufficiently large.
    It suffices to show that $\chi(G)\le s+t$, because this would imply that $G$ is homomorphic to $K_{s+t}$, which is $K_{1,s,t}$-free, \nolinebreak as \nolinebreak required.

    Since $3s\ge 2+2t$, we have
        $\delta(G)\ge
        \left(\frac{s}{1+s+t}+\eps \right)n
        \ge \left( \frac{2}{5}+\eps \right)n.$
    Note that for $r=3$ we have
    $\frac{3r-7}{3r-4} = \frac{2}{5}$.
    Thus, assuming $n\ge n_0(\varepsilon,s,t)$, we may apply \Cref{lem:AES-almost-partite} with $r=3$ and parameter
    $\frac{\varepsilon}{2}$ to obtain a partition $V(G) = A\cup B$ such that
        $\Delta(G[A]), \Delta(G[B])\le \varepsilon n/2.$
    This implies that, for any $(a,b)\in A\times B$,
    \begin{align*}
        d_B(a),d_A(b) \ge
        \left( \frac{s}{1+s+t}+\frac{\eps}{2} \right)n.
    \end{align*}
    Assume by symmetry that $|A|\le |B|$, we then have
    \begin{align}\label{ineq: rough ineq between A and B}
        \left( \frac{s}{1+s+t}+\frac{\eps}{2} \right )n\le
        |A|\le \frac{n}{2}\le |B|
        \le
        \left( \frac{t+1}{1+s+t}-\frac{\eps}{2} \right)n.
    \end{align}
    By (\ref{ineq: rough ineq between A and B}), there exists a (unique) integer $j$ with
    $0\le j\le \frac{t-s}{2}$
    such that
    \begin{align*}
        \frac{s+j}{1+s+t}n <
        |A|\le \frac{s+j+1}{1+s+t}n
        \ \text{    and    } \
        \frac{t-j}{1+s+t}n \le
        |B|< \frac{t-j+1}{1+s+t}n.
    \end{align*}
    Then for each $(a,b)\in A\times B$, we have
    \begin{align}\label{ineq:degrees in A and B}
        \frac{d_B(a)}{|B|}\ge
        \frac{s}{t-j+1} + \frac{\eps}{2}
        \ \text{ and } \
        \frac{d_A(b)}{|A|}\ge
        \frac{s}{s+j+1} + \frac{\eps}{2}.
    \end{align}
    We now apply Lemma \ref{lem:Brooks application} twice: once with $(X,Y) := (B,A)$ and once with $(X,Y) := (A,B)$. First, apply the lemma with
    $(X,Y) := (B,A)$ and $r := t-j+1 > s$. Note that $(r-s)^2 \leq (t+1-s)^2 \leq s$ (by the assumption of the theorem), so the conditions of Lemma \ref{lem:Brooks application} are satisfied. The lemma gives $\chi(G[B]) \leq r-1 = t-j$.

    Similarly, apply Lemma \ref{lem:Brooks application} with $(X,Y) = (A,B)$ and $r := s+j+1 > s$. Note that $(r- \nolinebreak s)^2 = (j+1)^2 \leq (t-s+1)^2 \leq s$, again using the assumption of the theorem. Lemma \ref{lem:Brooks application} gives $\chi(G[A]) \leq r-1 = s+j$.

    Combining the above, we get that $\chi(G) \leq \chi(G[A]) + \chi(G[B]) \leq (t-j) + (s+j) = s+t$, as required. This completes the proof.
\end{proof}

\section{Accumulation points of the homomorphism threshold}
\label{sec:hom-accumulation}

In this section we prove \Cref{thm: accumulation pts for hom}.  The case
$r=3$ follows directly from \Cref{thm: general lowbd for hom chi=3 and exact values}.
Indeed, for every $s\ge 2$ that theorem gives
    $\delta_{\mathrm{hom}}(K_{1,s,s})=\frac{s}{2s+1},$
and these values converge to $1/2$ as $s\to\infty$.

It remains to consider $r\ge 4$.  For integers $r\ge 4$ and $s\ge 2$,
let
\[
    H_{r,s}:=K_{1,s-1,s,\ldots,s},
\]
where the part of size $s$ occurs $r-2$ times.  We prove the following.

\begin{lemma}\label{lem:hom approaching to 1-1/r}
    For all integers $r\ge 4$ and $s\ge 2$,
        $\frac{r-2}{r-1}\cdot\frac{s-1}{s}
        \le \delta_{\mathrm{hom}}(H_{r,s})
        <\frac{r-2}{r-1}.$
\end{lemma}
\noindent
\Cref{lem:hom approaching to 1-1/r} immediately yields
\Cref{thm: accumulation pts for hom} (for $r \geq 4$), because $\delta_{\mathrm{hom}}(H_{r,s})$ converges to $\frac{r-2}{r-1}$ from below while never equaling $\frac{r-2}{r-1}$, implying that $\frac{r-2}{r-1}$ is an accumulation point.

\subsection{The upper bound in \Cref{lem:hom approaching to 1-1/r}}

    Put $q:=r-1$.
    Our goal is to show that there exists
    $\eta > 0$ depending only on $r,s$ such that
    $\delta_{\mathrm{hom}}(H_{r,s}) \leq
    \frac{r-2}{r-1}-\eta$. We will choose $\eta$ implicitly in the course of the proof.
    Let $G$ be an $n$-vertex graph with
    $\delta(G)\ge (\frac{q-1}{q}-\eta)n$ and with no copies of
    $H_{r,s} = K_{1,s-1,s,\dots,s}$. We assume throughout that $n$ is large enough where needed.
    Our goal is to show that $G$ is homomorphic to an $H_{r,s}$-free graph on $O(1)$ vertices.

    Note that
    $
    \frac{q-1}{q} - \eta > \frac{3q-4}{3q-1}+\eta
    $
    (provided that $\eta < \frac{1}{2q(3q-1)}$). Hence,
    $\delta(G) \geq
    ( \frac{3q-4}{3q-1} + \eta )n$.
    Therefore, by \Cref{lem:AES-almost-partite} (with $\varepsilon := \eta$), there is a partition
        $V(G)=V_1\cup\cdots\cup V_q$
    such that
         $\Delta(G[V_i])\le \eta n$
for every $i\in[q]$.
    We now prove the following simple claim:
    \begin{claim}\label{claim:accumulation point upper bound, sizes}
    The following hold for every $i\in[q]$:
    \begin{enumerate}
        \item 
        $|V_i|
        \le (\frac1q+2\eta)n.$
        \item for every $v\in V(G)\setminus V_i$, 
        $d_G(v,V_i)
        \ge \big(\frac1q-2(q-1)\eta\big)n.$
    \end{enumerate}
    \end{claim}
    \begin{poc}
    First we prove Item 1. If $V_i = \emptyset$ then there is nothing to prove. Otherwise, taking any $v \in V_i$, we have
        \[
        \left( \frac{q-1}{q} - \eta \right)n \leq \delta(G)\le d_G(v)\le n-|V_i|+\eta n,
        \]
    which rearranges to the required bound.
    Next, let $v\in V_h$ with
    $h \in [q] \setminus \{i\}$.
    Note that $d_G(v,V_h) \leq \eta n$.
    Now, using the bound in Item 1 for the $q-2$ parts $V_j$ with $j \neq i,h$, we obtain
    \begin{align*}
        d_G(v,V_i)
        \ge \delta(G)-d_G(v,V_h)
             -\sum_{j\in[q]\setminus\{i,h\}}|V_j|
        &\ge
        \left( \frac{q-1}{q} - \eta \right)n - \eta n -
        (q-2) \cdot \left( \frac1q+2\eta\right)n
        \\ &=
        \left(\frac1q-2(q-1)\eta\right)n.\qquad \qedhere
    \end{align*}
    \end{poc}

    By combining both Items of Claim \ref{claim:accumulation point upper bound, sizes}, we see that for every $i \in [q]$ and $v \in V(G) \setminus V_i$, it holds that
    $$
    \frac{d_G(v,V_i)}{|V_i|} \geq
        \frac{1-2q(q-1)\eta}{1+2q\eta}
        \ge 1-2q^2\eta.
    $$
    Thus, setting
        $\alpha:=2q^2\eta,$
    we record the fact that
    \begin{equation}\label{eq:cross-density-upper-Hrs}
        d_G(v,V_i) \geq (1-\alpha)|V_i|
        \qquad
        \text{for all }i \in[q]\text{ and }v\in V(G) \setminus V_i.
    \end{equation}

    We claim that every $G[V_i]$ is $K_{1,s-1}$-free.  Suppose otherwise;
    by symmetry, let $T\subseteq G[V_q]$ be a copy of $K_{1,s-1}$, so
    $|V(T)|=s$.  We shall choose sets        $U_j\subseteq V_j,$ $|U_j|=s$, $j\in[q-1]$,
    such that every two of the sets
        $V(T),U_1,\ldots,U_{q-1}$
    are completely joined in $G$. This would contradict the assumption that $G$ has no copies of $H_{r,s} = K_{1,s-1,s,\dots,s}$.

    Let $1 \leq j \leq q-1$ and suppose that $U_1,\ldots,U_{j-1}$ have already been chosen. Put
    \[
        S_j:=V(T)\cup U_1\cup\cdots\cup U_{j-1}.
    \]
    Every vertex of $S_j$ lies outside $V_j$, and $|S_j|\le qs$.
    Also, by choosing $\eta$ to satisfy
    $\eta \leq \frac{1}{4sq^3}$, we make sure that $qs\alpha \leq \frac{1}{2}$.
    Hence, by \eqref{eq:cross-density-upper-Hrs} and the union bound,
    \[
        \left|V_j\cap\bigcap_{x\in S_j}N_G(x)\right|
        \ge (1-|S_j|\alpha)|V_j|
        \ge \frac{|V_j|}{2}.
    \]
    Item 2 of Claim \ref{claim:accumulation point upper bound, sizes} in particular implies that $|V_j|=\Omega(n)$.
    Thus, for all sufficiently large $n$,
    we may choose $s$ vertices in $V_j$ which are adjacent to all vertices in $S_j$; choose $U_j$ to be the set of these $s$ vertices.  Continuing in this way
    gives the desired sets $U_1,\ldots,U_{q-1}$.
    As explained above, this gives a copy of $H_{r,s}$ in $G$, a contradiction. This proves our claim that $G[V_i]$ is $K_{1,s-1}$-free for every
    $i\in[q]$.

    Thus, for all $i \in [q]$ we have
    $\Delta(G[V_i])\le s-2$ and hence
    $\chi(G[V_i])\le s-1$ (using greedy coloring).
    Therefore, $\chi(G)\le q(s-1)=(r-1)(s-1),$
    so $G$ is homomorphic to
    $K_{(r-1)(s-1)}$.
    But this clique
    is $H_{r,s}$-free, since
    $|V(H_{r,s})|=(r-1)s>(r-1)(s-1).$
    So $G$ is homomorphic to an $H_{r,s}$-free graph on $O(1)$ vertices.
    We have therefore proved the desired upper bound
    \[
        \delta_{\mathrm{hom}}(H_{r,s})
        \le \frac{r-2}{r-1}-\eta
        <\frac{r-2}{r-1}.
    \]

\subsection{The lower bound in \Cref{lem:hom approaching to 1-1/r}}

We use the following special blow-up construction.

\begin{definition}\label{def:special blowup of multipartitecliques}
    Let $F$ be a graph and let $c:E(F)\to[2]$ be a two-coloring.  For an
    integer $m\ge 1$, the \emph{special $m$-blow-up} $F^{[m]}$ is defined
    as follows.  Each vertex $x\in V(F)$ is replaced by an independent set
        $U_x:=\{x_1,\ldots,x_m\}.$
    
    If $xy\in E(F)$ and $c(xy)=1$, then $F^{[m]}[U_x,U_y]$ is the matching
        $\{x_jy_j:j\in[m]\}.$
        
    If $xy\in E(F)$ and $c(xy)=2$, then $F^{[m]}[U_x,U_y]$ is the
    complement $\{x_jy_\ell:j,\ell\in[m],\ j\ne\ell\}.$
\end{definition}

\begin{proof}[Proof of the lower bound in
    \Cref{lem:hom approaching to 1-1/r}]
    Our goal is to construct, for any given $\varepsilon > 0$ and $C$, an $H_{r,s}$-free graph $G$ with
    $\delta(G) \geq
    (\frac{r-2}{r-1} \cdot \frac{s-1}{s} -\varepsilon)|V(G)|$ such that every homomorphic image of $G$ on at most $C$ vertices contains a copy of $H_{r,s}$.

    Put $q:=r-1$, so $q\ge 3$, and identify the vertex set of
    $K_{qs}$ with the grid
    \[
        [q]\times[s].
    \]
    We call $\{i\}\times[s]$ the $i$-th row.  Color the edges of $K_{qs}$ such that grid edges receive color 1 and all other edges receive color 2; i.e.,
    color the edge between the grid points $(i,a)$ and $(j,b)$ by
    \begin{equation}\label{eq:grid-colouring}
        c\bigl((i,a)(j,b)\bigr)
        :=
        \begin{cases}
            1, & \text{if }i=j\text{ or }a=b,\\
            2, & \text{if }i\ne j\text{ and }a\ne b.
        \end{cases}
    \end{equation}
    Let $\Gamma_m$ be the special $m$-blow-up of this colored copy of
    $K_{qs}$, as in Definition \ref{def:special blowup of multipartitecliques} (see \Cref{fig:lower-bound-r4-s3}).

    Every vertex of $K_{qs}$ has
    \[
        d_1=(s-1)+(q-1)=s+q-2
    \]
    incident edges of color $1$, and
    \[
        d_2=(q-1)(s-1)
    \]
    incident edges of color $2$.  Hence, $\Gamma_m$ is regular and
     \[\delta(\Gamma_m)=d_1+d_2(m-1)=(q-1)(s-1)m+
          \bigl(s+q-2-(q-1)(s-1)\bigr).
    \]
    In particular,
    \begin{equation}\label{eq:degree-ratio-Gamma-m}
        \frac{\delta(\Gamma_m)}{|V(\Gamma_m)|}
        \ge \frac{(q-1)(s-1)}{qs}
        -\frac{\bigl|s+q-2-(q-1)(s-1)\bigr|}{qsm}.
    \end{equation}
    Since $q=r-1$, the limiting degree ratio in
    \eqref{eq:degree-ratio-Gamma-m} as
    $m \rightarrow \infty$ is
    \[
        \frac{r-2}{r-1}\cdot\frac{s-1}{s}.
    \]

       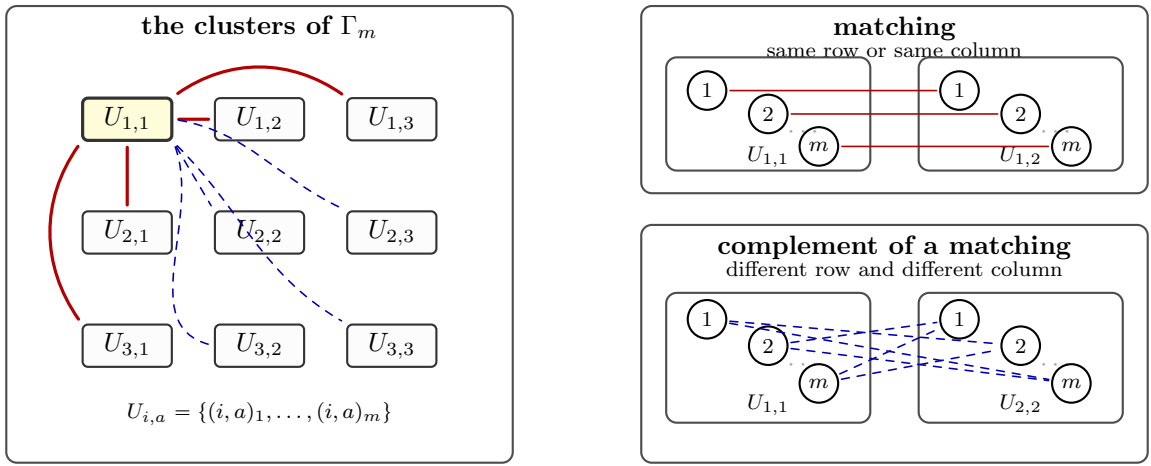
\begin{figure}[h]
\centering
\begin{tikzpicture}[
  font=\small,
  line cap=round,
  line join=round,
  panel/.style={draw=black!70, thick, rounded corners=4pt},
  cluster/.style={draw=black!80, thick, rounded corners=2pt,
                  minimum width=1.18cm, minimum height=0.56cm,
                  fill=gray!2, inner sep=1pt},
  focus/.style={cluster, very thick, fill=yellow!18},
  vertex/.style={draw, circle, thick, minimum size=5.1mm,
                 inner sep=0pt, font=\scriptsize},
  cone/.style={red!70!black, very thick, shorten <=2pt, shorten >=2pt},
  ctwo/.style={blue!70!black, dashed, semithick,
               shorten <=2pt, shorten >=2pt},
  match/.style={red!70!black, semithick,
                shorten <=1pt, shorten >=1pt},
  comp/.style={blue!70!black, dashed, semithick,
               shorten <=1pt, shorten >=1pt},
  faint/.style={text=gray!65}
]

\draw[panel] (-7.55,-2.95) rectangle (-0.85,3.10);
\node[font=\small\bfseries] at (-4.20,2.82)
  {the clusters of $\Gamma_m$};

\node[focus]   (U11) at (-5.95, 1.60) {$U_{1,1}$};
\node[cluster] (U12) at (-4.20, 1.60) {$U_{1,2}$};
\node[cluster] (U13) at (-2.45, 1.60) {$U_{1,3}$};

\node[cluster] (U21) at (-5.95, 0.10) {$U_{2,1}$};
\node[cluster] (U22) at (-4.20, 0.10) {$U_{2,2}$};
\node[cluster] (U23) at (-2.45, 0.10) {$U_{2,3}$};

\node[cluster] (U31) at (-5.95,-1.40) {$U_{3,1}$};
\node[cluster] (U32) at (-4.20,-1.40) {$U_{3,2}$};
\node[cluster] (U33) at (-2.45,-1.40) {$U_{3,3}$};

\draw[cone] (U11.east) -- (U12.west);
\draw[cone] (U11.north east) to[out=35,in=145] (U13.north west);
\draw[cone] (U11.south) -- (U21.north);
\draw[cone] (U11.south west) to[out=-125,in=125] (U31.north west);

\draw[ctwo] (U11.south east) -- (U22.north west);
\draw[ctwo] (U11.east) to[out=-10,in=155] (U23.north west);
\draw[ctwo] (U11.south east) to[out=-70,in=165] (U32.west);
\draw[ctwo] (U11.south east) to[out=-45,in=150] (U33.north west);

\node[font=\scriptsize] at (-4.20,-2.30)
  {$U_{i,a}=\{(i,a)_1,\ldots,(i,a)_m\}$};

\draw[panel] (0.85,0.62) rectangle (7.55,3.10);
\node[font=\small\bfseries] at (4.20,2.80)
  {matching};
\node[font=\scriptsize] at (4.20,2.52)
  {same row or same column};

\draw[panel] (1.18,0.92) rectangle (3.88,2.42);
\draw[panel] (4.52,0.92) rectangle (7.22,2.42);
\node[font=\scriptsize] at (2.53,1.10) {$U_{1,1}$};
\node[font=\scriptsize] at (5.87,1.10) {$U_{1,2}$};

\node[vertex] (A1) at (1.72,1.98) {$1$};
\node[vertex] (A2) at (2.53,1.67) {$2$};
\node[faint]        at (3.02,1.43) {$\cdots$};
\node[vertex] (Am) at (3.20,1.24) {$m$};

\node[vertex] (B1) at (5.06,1.98) {$1$};
\node[vertex] (B2) at (5.87,1.67) {$2$};
\node[faint]        at (6.36,1.43) {$\cdots$};
\node[vertex] (Bm) at (6.54,1.24) {$m$};

\draw[match] (A1) -- (B1);
\draw[match] (A2) -- (B2);
\draw[match] (Am) -- (Bm);

\draw[panel] (0.85,-2.95) rectangle (7.55,0.20);
\node[font=\small\bfseries] at (4.20,-0.10)
  {complement of a matching};
\node[font=\scriptsize] at (4.20,-0.38)
  {different row and different column};

\draw[panel] (1.18,-2.42) rectangle (3.88,-0.70);
\draw[panel] (4.52,-2.42) rectangle (7.22,-0.70);
\node[font=\scriptsize] at (2.53,-2.18) {$U_{1,1}$};
\node[font=\scriptsize] at (5.87,-2.18) {$U_{2,2}$};

\node[vertex] (C1) at (1.72,-1.05) {$1$};
\node[vertex] (C2) at (2.53,-1.40) {$2$};
\node[faint]        at (3.02,-1.66) {$\cdots$};
\node[vertex] (Cm) at (3.20,-1.90) {$m$};

\node[vertex] (D1) at (5.06,-1.05) {$1$};
\node[vertex] (D2) at (5.87,-1.40) {$2$};
\node[faint]        at (6.36,-1.66) {$\cdots$};
\node[vertex] (Dm) at (6.54,-1.90) {$m$};

\draw[comp] (C1) -- (D2);
\draw[comp] (C1) -- (Dm);
\draw[comp] (C2) -- (D1);
\draw[comp] (C2) -- (Dm);
\draw[comp] (Cm) -- (D1);
\draw[comp] (Cm) -- (D2);

\end{tikzpicture}
\caption{The lower-bound construction for $r=4$ and $s=3$. Here $q=r-1=3$, so the base graph is the colored copy of $K_9$ on the grid $[q]\times[s]$. }
\label{fig:lower-bound-r4-s3}
\end{figure}

    \noindent
    The following are the two key properties of the construction $\Gamma_m$.
    \begin{claim}\label{claim:Gamma-m-H-free}
        For every $m\ge 1$, the graph $\Gamma_m$ is
        $H_{r,s}$-free.
    \end{claim}
    \begin{claim}\label{claim:bounded-image-special-blowup}
        For every constant $C>0$, there exists an integer $m_0$ such that,
        for every $m\ge m_0$, every homomorphic image $\Gamma$ of
        $\Gamma_m$ with $|V(\Gamma)|\le C$ contains a copy of $K_{qs}$,
        and hence a copy of $H_{r,s}$.
    \end{claim}

    Let us complete the proof of the lower bound in \Cref{lem:hom approaching to 1-1/r} by using the above two claims.
    Given any $\varepsilon > 0$, we can choose $m$ large enough so that
    $\frac{\delta(\Gamma_m)}{|V(\Gamma_m)|} \geq
    \frac{r-2}{r-1}\cdot\frac{s-1}{s} - \varepsilon$, by \eqref{eq:degree-ratio-Gamma-m}. Also, $\Gamma_m$ is $H_{r,s}$-free but has no $H_{r,s}$-free homomorphic image on at most $C$ vertices, as required.
    \end{proof}
    \noindent
    In the next two sections, we prove Claims \ref{claim:Gamma-m-H-free} and \ref{claim:bounded-image-special-blowup}.

    \subsubsection{Proof of Claim \ref{claim:Gamma-m-H-free}}
        Suppose for a contradiction that $\Gamma_m$ contains a copy
        $H$ of $H_{r,s}$.
        Note that the vertices of $\Gamma_m$ have the form $(i,a)_j$ for $(i,a) \in [q] \times [s]$ and $j \in [m]$.
        For a vertex $v=(i,a)_j$ of $\Gamma_m$, write
        \[
            \pi(v):=(i,a)\in[q]\times[s]
            \quad\text{and}\quad
            \ell(v):=j\in[m].
        \]

        \begin{claim}\label{claim:row analysis}
            Consider any grid point
            $(i,a) \in [q] \times [s]$, and suppose that $V(H)$ contains at least two vertices with base point $(i,a)$, i.e., vertices $u,v \in V(H)$ with $\pi(u) = \pi(v) = (i,a)$. Then there is a partite class $P$ of
            $H \cong H_{r,s} = K_{1,s-1,s,\dots,s}$ such that all vertices of $H$ in row $i$ belong to $P$; in particular, $u,v \in P$.
        \end{claim}
        \begin{poc}
        For convenience, put $x := (i,a)$.
        Since there are no edges inside the cluster $U_x = \{x_j : j \in [m]\}$, and $u,v \in U_x$ by assumption, the vertices
        $u$ and $v$ belong to the same partite class of $H$, say $P$.
        If $w\in V(H)\setminus P$, then
        $w$ is adjacent to both $u$ and $v$. The base edge
        $\pi(w)x$ cannot have color $1$, since in that case adjacency to
        both $u$ and $v$ would force
        $\ell(w)=\ell(u)=\ell(v),$
        which is impossible because $u \neq v$.  Thus $\pi(w)x$ has color $2$. In
        particular, no vertex outside $P$ has its base point in row $i$ (because each such base point is connected to $x=(i,a)$ with an edge of color 1).
        Hence all vertices of $H$ in row $i$ belong to $P$, as required.
        \end{poc}

        Next, observe that each row $i$ contains at most $s$ vertices of
        $H$. Indeed, if the base points of the vertices of $H$
        in row $i$ are all distinct, this is immediate. And otherwise, there is $a \in [s]$ such that $H$ has at least two vertices $u,v$ with $\pi(u) = \pi(v) = (i,a)$. Then by Claim \ref{claim:row analysis}, all vertices of $H$ in row $i$ belong to some partite class of $H$, and hence there are at most $s$ such vertices.

        On the other hand, graph $H$ has
        $1+(s-1)+(q-1)s=qs$
        vertices and there are $q$ rows.  Consequently, every row
        contains exactly $s$ vertices of $H$.

        \begin{claim}\label{claim:pi is injective}
        For every two distinct vertices $u,v \in V(H)$ it holds that $\pi(u) \neq \pi(v)$.
        \end{claim}
        \begin{poc}
        Suppose for contradiction that
        $\pi(u)=\pi(v)=(i,a) =: x$ for some $(i,a) \in [q] \times [s]$.
        By Claim \ref{claim:row analysis}, there is a partite class $P$ of $H$ such that $u,v \in P$ and furthermore, the vertices of $H$ belonging to row $i$ are precisely the vertices in $P$.
        Thus $P$ is one
        of the partite classes of size $s$. Fix any $i'\ne i$ and consider row $i'$.  All $s$ vertices of $H$ in row $i'$ lie
        outside $P$. As observed above, their base points must be joined
        to $x$ by color-$2$ edges, because
        if $w \in V(H)$ is such that $\pi(w)x$ has color 1, then $w$ cannot be adjacent to both $u,v$.
        It follows that $\pi(w) \neq (i',a)$ for every $w \in V(H)$, because the edge between $x = (i,a)$ and $(i',a)$ has color 1.
        We conclude that the $s$ vertices of $H$ in row $i'$ are placed on at most $s-1$ base points (because having $\pi(w) = (i',a)$ is impossible).
        Hence, $H$ has two vertices which have the same base point in row $i'$.
        Now, by Claim \ref{claim:row analysis}, the vertices of $H$ in row $i'$ must all belong to the same partite class of $H$, and this partite class has size $s$ (because $H$ has $s$ vertices in each row).
        Since this holds for all $i' \in [q] \setminus \{i\}$, we conclude that
        $H$ has $q$ partite classes of size $s$. But this is impossible,
        because $H_{r,s}$ has only
        $r-2=q-1$ partite classes of size $s$.
        \end{poc}

        Thus, the map $\pi$ is injective on $V(H)$.
        Since both $V(H)$ and $[q]\times[s]$ have size $qs$,
        the map $\pi$ is in fact a bijection.

        Denote the partite classes of $H \cong H_{r,s}$ by
        $P_0,P_1,P_2,\ldots,P_q,$
        where
        \[
            |P_0|=1,\qquad |P_1|=s-1,
            \qquad |P_i|=s\quad(2\le i\le q).
        \]
        Write $P_0=\{z\}$, and put
        \[
            (i_0,a_0) := \pi(z),\qquad
            j_0 := \ell(z).
        \]
        Also, let
        $$R_{i_0} := \{(i_0,a) : a \in [s] \setminus \{a_0\}\} \; \; \; \; \;
        C_{a_0} := \{(i,a_0) : i \in [q] \setminus \{i_0\}\}
        $$
        denote the $i_0$'th row without the point $(i_0,a_0)$ and the $a_0$'th column without the point $(i_0,a_0)$, respectively.
        By the bijectivity of $\pi$, for every
        $x \in R_{i_0} \cup C_{a_0}$ there is $v_x \in V(H)$ with $\pi(v_x) = x$.
        Also, the vertices $z,v_x$ are adjacent in $H$ because $z$ is adjacent to all vertices of $H$ (since $P_0 = \{z\}$ is the partite class of size 1).
        On the other hand, each
        $x \in R_{i_0} \cup C_{a_0}$ is connected to $(i_0,a_0)$ in color 1. This means that $\ell(v_x) = \ell(z) = j_0$ for each $x \in R_{i_0} \cup C_{a_0}$.

        Next, consider any $x \in R_{i_0}$ and $y \in C_{a_0}$. The base points $x$ and $y$ differ in both
        coordinates, so the base edge $xy$ has color $2$. But since $\ell(v_x) = \ell(v_y)$, this means that $v_x,v_y$ are not adjacent, and hence belong to the same partite class of $H$. Varying $x$ and $y$, we conclude that all
        \[
            (s-1)+(q-1)=s+q-2
        \]
        vertices $v_x$ for $x\in R_{i_0}\cup C_{a_0}$. belong
        to the same partite class. But $q\ge 3$, so $s+q-2>s$, whereas every
        partite class of $H \cong H_{r,s}$ has size at most $s$. This final
        contradiction proves Claim \ref{claim:Gamma-m-H-free}.

    \subsubsection{Proof of Claim \ref{claim:bounded-image-special-blowup}}

        Let            $\varphi:\Gamma_m\rightarrow\Gamma$
        be a homomorphism with $|V(\Gamma)|\le C$.  For every $j\in[m]$,
        consider the ordered $qs$-tuple
        \[
            \mathbf a_j
            :=\bigl(\varphi(x_j)\bigr)_{x\in[q]\times[s]}.
        \]
        There are at most $C^{qs}$ possible choices for $\mathbf a_j$. Hence, if
        \[
            m \geq m_0:=\lfloor C^{qs}\rfloor+1,
        \]
        then there exist distinct $j,k\in[m]$ such
        that $\mathbf a_j=\mathbf a_k$.

        For each base point $x\in[q]\times[s]$, put
        \[
            u_x:=\varphi(x_j)=\varphi(x_k).
        \]
        We claim that $\{u_x:x\in[q]\times[s]\}$ is a clique in
        $\Gamma$.  Fix distinct base points
        $x,y \in [q] \times [s]$.
        If $c(xy)=1$, then
        $x_jy_j\in E(\Gamma_m)$, and hence $u_xu_y = \varphi(x_j)\varphi(y_j) \in E(\Gamma)$ (as $\varphi$ is a homomorphism).
        And if
        $c(xy)=2$, then $x_jy_k\in E(\Gamma_m)$ because $j \neq k$, and
        again $u_xu_y = \varphi(x_j)\varphi(y_k) \in E(\Gamma)$. Thus every two distinct vertices in $\{u_x:x\in[q]\times[s]\}$ are adjacent, as claimed. In particular, these vertices are pairwise
        distinct and form a copy of $K_{qs}$ in $\Gamma$, as required.

\section{Proof of \Cref{thm:complete results on tripartite graphs on VC}}\label{sec:Krst VC}
The proof of \Cref{thm:complete results on tripartite graphs on VC} is split into the following four theorems:
\begin{theorem}\label{thm:K(r,s,t) r>=2}
$\delta_{\mathrm{VC}}(K_{r,s,t}) = \frac{1}{2}$ for all $2\le r\le s\le t$.
\end{theorem}
\begin{theorem}\label{thm:K(1,s,t), = 1/2}
$\delta_{\mathrm{VC}}(K_{1,s,t}) = \frac{1}{2}$
for all $s,t$ with $3\le s\le t\le 2s-3$,
\end{theorem}
\begin{theorem}\label{thm:K(1,s,t), < 1/2}
$\delta_{\mathrm{VC}}(K_{1,s,t}) \geq \frac{s}{2s+1}$
for all $1 \leq s \leq t$.
\end{theorem}
\begin{theorem}\label{thm:K(1,s,t) upper bound}
    $\delta_{\mathrm{VC}}(K_{1,s,t}) \leq
    \frac{s}{2s+1}$ for all $1 \leq s \leq t$
    with $2s-2\le t$.
\end{theorem}

It is easy to see that the above four theorems together give Theorem
\ref{thm:complete results on tripartite graphs on VC}.
Note that trivially $\delta_{\mathrm{VC}}(H) \leq \frac{1}{2}$ for every 3-chromatic graph $H$ (by the Erd\H{o}s-Stone theorem), so for Theorems \ref{thm:K(r,s,t) r>=2} and \ref{thm:K(1,s,t), = 1/2} we only need to prove the lower bound.
\begin{proof}[Proof of Theorem \ref{thm:K(r,s,t) r>=2}]
For each $d \geq 1$ and $\varepsilon > 0$, we need to construct an $n$-vertex maximal $K_{r,s,t}$-free  graph $G$ with VC-dimension at least $d$ and
$\delta(G) \geq (\frac{1}{2} - \varepsilon) n$. To this end, start with an $m \times m$ bipartite graph $F$ with VC-dimension at least $d$ (where $m$ is an integer depending on $d$). For convenience, identify both parts of $F$ with the set $[m]$. Take disjoint sets
$\{a_i\},B_i,C_i,D_i$, $i = 1,\dots,m$, where
$|B_i| = s$,
$|C_i| = r-1$ and $|D_i| = t$. Add all edges between $\{a_i\}$ and $B_i$ and between $C_i$ and $D_i$ for each $1 \leq i \leq m$. For convenience, set
$A := \{a_1,\dots,a_m\}$,
$B := \bigcup_{i=1}^m B_i$,
$C := \bigcup_{i=1}^m C_i$ and
$D := \bigcup_{i=1}^m D_i$.
Thus, $A \cup B$ induces a disjoint union of copies of $K_{1,s}$, and $C \cup D$ induces a disjoint union of copies of $K_{r-1,t}$.
We now define the edges between the sets $A \cup B$ and $C \cup D$, as follows. For each $i \in [m]$, fix a vertex $d_i \in D_i$ (arbitrarily).
\begin{itemize}
    \item For each pair $(i,j) \in [m]^2$ with
    $ij \notin E(F)$, add the complete bipartite graphs
    $$
    (B_i,C_j),(B_i,D_j),(a_i,D_j \setminus \{d_j\}).
    $$
    \item For each $(i,j) \in [m]^2$ with
    $ij \in E(F)$, add the edge $a_id_j$.
\end{itemize}
These are all the edges between $A \cup B$ and $C \cup D$.
Denote the resulting graph by $G_0$.
Note that for all $ij \notin E(F)$, the set
$\{a_i\} \cup B_i \cup C_j \cup D_j$ induces a copy of $K_{r,s,t} - e$, with the missing edge being $a_id_j$; the parts of this $K_{r,s,t}$ are $\{a_i\} \cup C_j$ (of size $r$), $B_i$ (of size $s$), and $D_j$ (of size $t$). This means that adding the edge $a_id_j$ would create a copy of $K_{r,s,t}$. Secondly, observe that for $i,j \in [m]$, we have that $a_id_j \in E(G_0)$ if and only if $(i,j) \in E(F)$. Thus, the defined graph contains a bi-induced copy of $F$, and hence has VC-dimension at least $d$.

Next, add two sets $X,Y$ of size $N$ with $N \gg m,1/\varepsilon$, and add the complete bipartite graphs
$$
(X, C \cup D), (Y, A \cup B), (X,Y).
$$
The resulting graph, denoted $G'$, has $n := 2N + m(r+s+t)$ vertices and
$\delta(G') \geq N \geq (\frac{1}{2} - \varepsilon)n$, provided that $N$ is large enough in terms of $m,\varepsilon$.
The main claim is as follows:
\begin{claim}
    $G'$ is $K_{r,s,t}$-free.
\end{claim}
\begin{poc}
    Suppose that $G'$ contains a copy of $K_{r,s,t}$ with parts
    $R,S,T$, and put $V:=R\cup S\cup T$.  The neighborhood of every
    vertex of $K_{r,s,t}$ contains a copy of $K_{r,s}$.  For $x\in X$,
    however, $G'[N(x)]$ is a disjoint union of copies of $K_{r-1,t}$ and
    isolated vertices; for $y\in Y$, it is a disjoint union of copies of
    $K_{1,s}$ and isolated vertices.  Both graphs are $K_{r,s}$-free, so
    $V\cap(X\cup Y)=\emptyset$.

    Put $L_i:=\{a_i\}\cup B_i$ and $R_j':=C_j\cup D_j$.  We claim that
    $V$ meets at most one set $L_i$.  Indeed, if
    $u\in L_i\cap V$ and $v\in L_j\cap V$ with $i\ne j$, then $uv$ is a
    non-edge, so $u$ and $v$ lie in the same part of the copy.  Their
    common neighborhood inside the copy therefore contains $K_{r,s}$.
    On the other hand,
    $N_{G'}(u)\cap N_{G'}(v)\subseteq C\cup D\cup Y$, whose induced graph
    is a disjoint union of copies of $K_{r-1,t}$ and isolated vertices, a
    contradiction.  The same argument, with the two sides interchanged,
    shows that $V$ meets at most one set $R_j'$: the common neighborhood
    of vertices from two distinct such blocks is contained in
    $A\cup B\cup X$, a disjoint union of copies of $K_{1,s}$ and isolated
    vertices.

    Hence $V\subseteq L_i\cup R_j'$ for some $i,j\in[m]$.  This set has
    exactly $r+s+t$ vertices, so any copy of $K_{r,s,t}$ in it would use
    every vertex.  In any such spanning copy, the endpoints of every
    non-edge of the host must lie in the same part.  If $ij\notin E(F)$,
    the induced graph is $K_{r,s,t}$ with the single edge $a_id_j$
    deleted.  Each of the three original partite classes must lie in one
    part, while the additional non-edge $a_id_j$ forces the classes
    containing $a_i$ and $d_j$ into the same part; their union has more
    than $t$ vertices, which is impossible.  If $ij\in E(F)$, there is
    only one edge between $L_i$ and $R_j'$.  The graph of non-edges across
    this cut is connected, so all vertices would have to lie in one part,
    again impossible.  Thus $G'$ is $K_{r,s,t}$-free.
\end{poc}
Add edges to $G'$ until the resulting graph, denoted $G$, is maximal $K_{r,s,t}$-free. Then clearly $\delta(G) \geq \delta(G') \geq
(\frac{1}{2}-\varepsilon)n$. Also, we saw that the edges $a_id_j$ for $ij \notin E(F)$ cannot be added without creating a $K_{r,s,t}$. Hence,
$\{a_1,\dots,a_m\},\{d_1,\dots,d_m\}$ still form a bi-induced copy of $F$ in $G$, hence 
$\VCdim(G) \geq d$. This completes the proof of the theorem.
\end{proof}

\begin{proof}[Proof of Theorem \ref{thm:K(1,s,t), = 1/2}]
The proof idea is similar to that of Theorem \ref{thm:K(r,s,t) r>=2}. Given any $d \geq 1$ and $\varepsilon > 0$, let $F$ be an $m \times m$ bipartite graph with VC-dimension at least $d$, where $m$ is bounded above by a function of $d$. We identify both sides of $F$ with $[m]$. Take disjoint sets
$\{a_i\},\{b_i\},C_i,D_i,E_i$ for $i \in [m]$, where
$|C_i| = s-2$, $|D_i| = s-1$ and $|E_i| = t-s+2$.
For convenience, put $A := \{a_1,\dots,a_m\}$, $B = \{b_1,\dots,b_m\}$, $C = \bigcup_{i=1}^m C_i$,
$D = \bigcup_{i=1}^m D_i$ and $E = \bigcup_{i=1}^m E_i$.

For each $i \in [m]$, add the complete tripartite graph $(a_i,b_i,C_i)$ and the complete bipartite graph $(D_i,E_i)$. Thus, $A \cup B \cup C$ is a disjoint union of copies of $K_{1,1,s-2}$, and $D \cup E$ is a disjoint union of copies of $K_{s-1,t-s+2}$. We now define the edges between the sets
$A \cup B \cup C$ and $D \cup E$. For each $i \in [m]$, fix a vertex $e_i \in E_i$ (arbitrarily).
\begin{itemize}
    \item For each pair $(i,j) \in [m]^2$ with
    $ij \notin E(F)$, add the complete bipartite graphs
    $$
    (b_i,E_j),(C_i,D_j),(a_i,D_j),(a_i,E_j \setminus \{e_j\}).
    $$
    \item For each $(i,j) \in [m]^2$ with
    $ij \in E(F)$, add the edge $a_ie_j$.
\end{itemize}
Denote the resulting graph by $G_0$.
Note that for all $ij \notin E(F)$, the set
$\{a_i,b_i\} \cup C_i \cup D_j \cup E_j$ induces a copy of $K_{1,s,t} - e$, with parts $\{a_i\}$, $\{b_i\} \cup D_j$ (of size $s$) and $C_i \cup E_j$ (of size $t$), and with missing edge $a_ie_j$. Also, for $i,j \in [m]$ we have that $a_ie_j \in E(G_0)$ if and only if $ij \in E(F)$, meaning that
$\{a_1,\dots,a_m\},\{e_1,\dots,e_m\}$ give a bi-induced copy of $F$ in $G_0$.

To increase the minimum degree, we again add two sets $X,Y$ of size $N$ and the complete bipartite graphs
$$
(X, D \cup E), (Y, A \cup B \cup C), (X,Y).
$$
The resulting graph $G'$ has $n := 2N + m(1+s+t)$ vertices and minimum degree $\delta(G') \geq (\frac{1}{2}-\varepsilon)n$, provided that $N$ is large enough as a function of $m,1/\varepsilon$.
\begin{claim}\label{claim:K(1,s,t)-free}
    $G'$ is $K_{1,s,t}$-free.
\end{claim}
\begin{poc}
    Suppose that $G'$ contains a copy of $K_{1,s,t}$ with vertex set
    $V$.  The neighborhood of every vertex of $K_{1,s,t}$ contains a
    copy of $K_{1,s}$.  If $x\in X$, then $G'[N(x)]$ is a disjoint union
    of copies of $K_{s-1,t-s+2}$ and isolated vertices; its maximum degree
    is at most $s-1$ because $t\le2s-3$.  If $y\in Y$, then $G'[N(y)]$ is
    a disjoint union of copies of $K_{1,1,s-2}$ and isolated vertices,
    again of maximum degree at most $s-1$.  Hence
    $V\cap(X\cup Y)=\emptyset$.

    Put $L_i:=\{a_i,b_i\}\cup C_i$ and $R_j':=D_j\cup E_j$.  If
    $u\in L_i\cap V$ and $v\in L_j\cap V$ with $i\ne j$, then $uv$ is a
    non-edge, so $u,v$ lie in the same part of the copy and their common
    neighborhood contains $K_{1,s}$.  But
    $N_{G'}(u)\cap N_{G'}(v)\subseteq D\cup E\cup Y$, whose induced graph
    is a disjoint union of copies of $K_{s-1,t-s+2}$ and isolated
    vertices.  This is impossible.  Thus $V$ meets at most one left
    block $L_i$.  Similarly, vertices from two distinct right blocks
    $R_i',R_j'$ have common neighborhood contained in
    $A\cup B\cup C\cup X$, a disjoint union of copies of $K_{1,1,s-2}$
    and isolated vertices.  Hence $V$ meets at most one right block.

    Therefore $V\subseteq L_i\cup R_j'$ for some $i,j\in[m]$.  This set
    has exactly $1+s+t$ vertices.  In a spanning complete tripartite copy,
    every connected component of the host's non-edge graph lies within a
    single part.  If $ij\notin E(F)$, the induced graph is $K_{1,s,t}$
    with the edge $a_ie_j$ deleted.  The non-edges inside the $t$-vertex
    class, together with $a_ie_j$, force the singleton and the entire
    $t$-class into one part, which is impossible.  If $ij\in E(F)$, the
    only edge between $L_i$ and $R_j'$ is $a_ie_j$; the non-edge graph
    across the cut is connected, so all vertices would have to lie in one
    part.  This is again impossible.  Thus $G'$ is $K_{1,s,t}$-free.
\end{poc}
As in the proof of Theorem \ref{thm:K(r,s,t) r>=2}, let $G$ be a maximal $K_{1,s,t}$-free supergraph of $G'$. We have $\VCdim(G) \geq d$, and
$\delta(G) \geq (\frac{1}{2}-\varepsilon)n$, as required.
\end{proof}
\begin{proof}[Proof of Theorem \ref{thm:K(1,s,t), < 1/2}]
Our goal is to construct, for any given $d \geq 1$ and $\varepsilon > 0$, an $n$-vertex maximal $K_{1,s,t}$-free graph $G$ with
$\delta(G) \geq (\frac{s}{2s+1} - \varepsilon)n$ and $\VCdim(G) \geq d$. Let $F$ be an $m \times m$ bipartite graph with $\VCdim(F) \geq d$, where $m$ depends only on $d$. We identify both sides of $F$ with $[m]$.
Take disjoint sets $A_i,B_i$, $i = 1,\dots,m$, where $|A_i| = s+1$ and $|B_i| = t$. Write $A_i = \{a_i^0,\dots,a_i^s\}$.
For convenience, put $A := \bigcup_{i=1}^m A_i$ and
$B := \bigcup_{i=1}^m B_i$.
Also, add disjoint sets $X,Y_0,\dots,Y_{s}$ with
$|X| = sN$ and $|Y_i| = N$ for all $0 \leq i \leq s$, where $N \gg m,1/\varepsilon$. Put
$Y := \bigcup_{i=0}^{s}Y_i$.
We now define a graph $G'$ on the vertex-set
$
A \cup B \cup X \cup Y,
$
as follows. For each $i \in [m]$, fix a vertex $b_i \in B_i$ (arbitrarily). We note that the construction here is somewhat different from those used in the proofs of Theorems \nolinebreak \ref{thm:K(r,s,t) r>=2} \nolinebreak and \nolinebreak \ref{thm:K(1,s,t), = 1/2}.
\begin{itemize}
    \item Add the complete bipartite graphs $(X,B),(X,Y)$.
    \item For every $i \in [m]$ and all $0 \leq j,k \leq s$, add the complete bipartite graph $(a_i^j,Y_k)$ if and only if $j \neq k$. This makes sure that there is no $y \in Y$ which is adjacent to all vertices of $A_i$. Also,
    $d_Y(a) = sN$ for all $a \in A$.
    \item For each $i \in [m]$, add the complete tripartite graph
    $(a_i^0, A_i \setminus \{a_i^0\}, B_i \setminus \{b_i\})$. Thus, the
    $A_i \cup \nolinebreak (B_i \setminus \{b_i\})$ induces a $K_{1,s,t-1}$.
    \item For all $(i,j) \in [m]^2$, if $ij \notin E(F)$ then add the complete bipartite graph
    $(A_i \setminus \{a_i^0\},b_j)$, and if
    $ij \in E(F)$ then add the edge $a_i^0b_j$.
\end{itemize}
The resulting graph is $G'$.
We have $n := |V(G')| = (2s+1)N + m(1+s+t)$ and
$\delta(G') \geq sN \geq (\frac{s}{2s+1} - \varepsilon)n$, provided that $N \gg m,1/\varepsilon$.
Next, note that for all $(i,j) \in [m]^2$, we have that $a_i^0b_j \in E(G')$ if and only if $ij \in E(F)$.
Also, observe that for each $(i,j) \notin E(F)$, adding the edge $a_i^0b_j$ results in a copy of $K_{1,s,t}$ with parts $\{a_i^0\}$, $A_i \setminus \{a_i^0\}$ (of size $s$) and $(B_i \setminus \{b_i\}) \cup \{b_j\}$ (of size $t$). This means that in every $K_{1,s,t}$-free supergraph $G$ of $G'$, the vertices
$\{a_1^0,\dots,a_m^0\},\{b_1,\dots,b_m\}$ form a bi-induced copy of $F$, and hence
$\VCdim(G) \geq d$.
\begin{claim}
    $G'$ is $K_{1,s,t}$-free.
\end{claim}
\begin{poc}
    We need to check that for each $v \in V(G')$, $N_{G'}(v)$ is $K_{s,t}$-free. We consider the following cases:
    \begin{itemize}
        \item If $v \in X$ then
        $N_{G'}(v) = B \cup Y$ is independent.
        \item Suppose that $v \in Y$. If $v \in Y_0$ then $N_{G'}(v)$ is independent, and if $v \in Y_i$ for some $i \in [s]$ then $N_{G'}(v)$ consists of isolated vertices and disjoint copies of $K_{1,s-1}$, hence is $K_{s,t}$-free.
        \item Suppose that $v \in B_j \setminus \{b_j\}$ for some $j \in [m]$. Then $N_{G'}(v)$ consists of isolated vertices and one $K_{1,s}$-copy (induced by $A_j$). Thus, $N_{G'}(v)$ is $K_{s,t}$-free unless $s=t=1$. However, in that case $B_j \setminus \{b_j\} = \emptyset$ (as $|B_j| = t$), so this case is impossible.
        \item Suppose that $v = b_j$ for some
        $j \in [m]$. Then $N_{G'}(v)$ is independent, because for each $i \in [m]$, $v = b_j$ cannot have neighbors in both of the sets
        $\{a_i^0\},A_i \setminus \{a_i^0\}$.
        \item Suppose that
        $v \in A_i \setminus \{a_i^0\}$ for some
        $i \in [m]$;
        say $v = a_i^j$ for $j \in [s]$.
        Then $N_{G'}(v)$ consists of isolated vertices and a star with center $a_i^0$ and leaf-set
        $(B_i \setminus \{b_i\}) \cup \bigcup_{k \neq 0,j}Y_k$.
        Hence, if $s \geq 2$ then $N_{G'}(v)$ is clearly $K_{s,t}$-free. And if $s = 1$, then $j=1$ and the union $\bigcup_{k \neq 0,j}Y_k$ is empty (since there is no such index $k \in \{0,\dots,s\} = \{0,1\}$). Therefore, the star in $N_{G'}(v)$ has $t-1$ leaves, and hence $N_{G'}(v)$ is $K_{1,t}$-free.
        \item Suppose that $v = a_i^0$ for some $i \in [m]$.
        The only edges in $N_{G'}(v)$ are between $A_i \setminus \{a_i^0\}$ and
        $(B_i \setminus \{b_i\}) \cup \bigcup_{k=1}^s Y_k$. In particular, $N_{G'}(v)$ induces a bipartite graph with one part (namely $A_i \setminus \{a_i^0\}$) of size $s$.
        Hence, in order for $N_{G'}(v)$ to contain $K_{s,t}$, the vertices in
        $A_i \setminus \{a_i^0\}$ must have at least $t$ common neighbors. But these vertices have no common neighbors in $Y$, and
        $|B_i \setminus \{b_i\}| = t-1$. Hence, $N_{G'}(v)$ is $K_{s,t}$-free. \qedhere
    \end{itemize}
\end{poc}
Let $G$ be a maximal $K_{1,s,t}$-free supergraph of $G'$. Then $\delta(G) \geq (\frac{s}{2s+1} - \varepsilon)n$, and we already saw above that
$\VCdim(G) \geq d$, completing the proof.
\end{proof}

\begin{proof}[Proof of Theorem \ref{thm:K(1,s,t) upper bound}]
Let $\varepsilon > 0$, and let $G$ be a maximal $K_{1,s,t}$-free $n$-vertex graph with
$\delta(G) \geq ( \frac{s}{2s+1} + \varepsilon )n$. Our goal is to show that
$\VCdim(G) \leq O_{s,t,\varepsilon}(1)$, provided that $n \geq n_0(\varepsilon)$.

We handle separately the cases $s=1$ and $s\ge2$.  Suppose first that
$s=1$.  We claim that $G$ is triangle-free.  If $xyz$ were a triangle,
then
\[
\sum_{\{u,v\}\in\binom{\{x,y,z\}}2}|N(u)\cap N(v)|
\ge d(x)+d(y)+d(z)-n>3\varepsilon n,
\]
so some edge of the triangle would have at least $\varepsilon n\ge t$
common neighbours, giving $K_{1,1,t}$.  Moreover, $G$ is maximal
triangle-free: adding a missing edge and creating a $K_{1,1,t}$ would also
create a triangle containing that edge.  Since
$\delta_{\mathrm{VC}}(K_3)=1/3$~\cite{huang_interpolating_2025}, it follows
that $\VCdim(G)=O_{t,\varepsilon}(1)$.

Suppose now that $s\ge2$.  Since
$\delta(G)\ge(2/5+\varepsilon)n$, \Cref{lem:AES-almost-partite} gives a
partition $V(G)=A\cup B$ with
$\Delta(G[A]),\Delta(G[B])\le\varepsilon n/2$.  Consequently
$d_B(a),d_A(b)
\ge
\left(\frac{s}{2s+1}+\frac{\varepsilon}{2}\right)n$
for all $a\in A$ and $b\in B$.  Each of $A,B$ therefore has size at most
$(s+1)n/(2s+1)$, and hence every vertex has relative degree at least
$s/(s+1)+\varepsilon/2$ into the opposite part.  Thus any $s+1$ vertices
in either part have $\Omega(n)$ common neighbours in the other part.  It
follows that both $G[A]$ and $G[B]$ are $K_{1,s}$-free.

We use the following elementary observation: if $t\ge2s-2$, then every
bipartition of $V(K_{1,s,t})$ has a side containing $K_{1,s}$.  Indeed,
color the two sides red and blue.  If one color occurs only in the
$t$-vertex class, then the other color contains the singleton together
with the whole $s$-vertex class and hence contains $K_{1,s}$.  Otherwise,
each color used in the $t$-class also occurs outside it, so each has at
most $s-1$ vertices in the $t$-class.  Equality for both colors would
force the singleton and the $s$-vertex class to be monochromatic in
opposite colors, in which case a vertex of the $t$-class is the centre of
a monochromatic $K_{1,s}$.  Hence a bipartition with neither side
containing $K_{1,s}$ would imply $t\le2s-3$.

Let $G'$ be obtained from $G$ by completing the pair $(A,B)$.  The
observation and the $K_{1,s}$-freeness of $G[A]$ and $G[B]$ show that $G'$
is $K_{1,s,t}$-free.  Maximality therefore forces $G=G'$, so $G$ is the
join of $G[A]$ and $G[B]$.  Each side has VC-dimension at most $s-1$, and
\Cref{fact:VC-of-join} gives $\VCdim(G)\le s$.
\end{proof}

\section{Proof of Theorem \ref{thm:strict lowbd for delta_VC}}\label{sec:VC lower bound}
We need to recall the notion of the {\em core} of a graph.
\begin{definition}\label{def:core}
    The {\em core} of a graph $H$, denoted $\mathrm{core}(H)$, is the smallest subgraph $K$ of $H$ (in terms of the number of vertices) such that there is a homomorphism from $H$ to $K$.
\end{definition}

It is not hard to see that the core of a graph $H$ is unique up to isomorphism, i.e., that if two subgraphs $K_1,K_2$ of $H$ satisfy Definition \ref{def:core} then $K_1,K_2$ are isomorphic.

A graph $H$ is called a {\em core} if
$\mathrm{core}(H) = H$. The following three statements are well-known.
\begin{lemma}\label{lem:core}
For every graph $H$, the graph
$K := \mathrm{core}(H)$ is a core.
\end{lemma}
\begin{proof}
    If not, then there is a homomorphism $\varphi$ from $K$ to a proper subgraph $K'$ of $K$. But then there is also a homomorphism from $H$ to $K'$ (obtained by composing $\varphi$ with a homomorphism from $H$ to $K$). However, this contradicts the minimality in the definition of $\mathrm{core}(H)$.
\end{proof}
\begin{lemma}\label{lem:core homomorphism poset}
Suppose that $H_1,H_2$ are cores and that there exist homomorphisms $\varphi : H_1 \rightarrow H_2$ and $\psi : H_2 \rightarrow H_1$. Then $H_1,H_2$ are isomorphic.
\end{lemma}
\begin{proof}
    Every endomorphism of a finite core is an automorphism, since a
    non-surjective endomorphism would map the core to a proper subgraph.
    Hence both $\psi\circ\varphi$ and $\varphi\circ\psi$ are
    automorphisms.  In particular, $\varphi$ is bijective and
    $\varphi^{-1}=(\psi\circ\varphi)^{-1}\circ\psi$ is a homomorphism.
    Thus $\varphi$ is an isomorphism.
\end{proof}

\begin{lemma}\label{lem:core homomorphism}
Let $K$ be a core with $V(K) = [k]$. Let $G$ be a graph with vertex partition $V(G) = V_1 \cup \dots \cup V_k$ such that the map $V_i \mapsto i$ is a homomorphism from $G$ to $K$. Then every homomorphism $\varphi$ from $K$ to $G$ is injective, and
$|\mathrm{Im}(\varphi) \cap V_i| = 1$ for all $i \in [k]$.
\end{lemma}
\begin{proof}
    Let $p:G\to K$ be the homomorphism that maps $V_i$ to $i$.  Then
    $p\circ\varphi$ is an endomorphism of the core $K$, and hence an
    automorphism.  Therefore $\varphi$ is injective and the image of
    $p\circ\varphi$ meets every vertex of $K$.  Since $K$ has $k$
    vertices, $\mathrm{Im}(\varphi)$ meets each part $V_i$ exactly once.
\end{proof}

\noindent
Finally, we need the following simple graph-theoretic lemma.

\begin{lemma}\label{lem:orientation}
    Let $G$ be a graph and suppose that every connected component of $G$ contains a cycle. Then there is an orientation $D$ of a spanning subgraph of $G$ such that $d^+_D(v) = 1$
    for every $v \in V(G)$.
\end{lemma}
\begin{proof}
    In each connected component choose a spanning unicyclic subgraph.
    Orient its unique cycle cyclically and orient every remaining edge
    towards the cycle.  Each vertex then has out-degree exactly one.
\end{proof}

\begin{proof}[Proof of Theorem \ref{thm:strict lowbd for delta_VC}]
Let $H$ be a graph with $r := \chi(H) \geq 3$.
We begin with some definitions. Let $\mathcal{H}$ be the set of all graphs $H'$ for which there is a partition
$V(H) = A_1 \cup \dots \cup A_{r-3} \cup B$ where $A_1,\dots,A_{r-3}$ are independent and
$H[B] \cong H'$.
For each $H' \in \mathcal{H}$ we have
$\chi(H') \geq 3$ (as $\chi(H) = r$).
Let
$\mathcal{K} := \{\mathrm{core}(H') : H' \in \mathcal{H}\}$ be the set of cores of graphs in $\mathcal{H}$. We consider $\mathcal{K}$ as a set of unlabeled graphs, meaning that we identify graphs which are isomorphic. By Lemma \ref{lem:core}, all graphs in $\mathcal{K}$ are cores. Now consider the relation on $\mathcal{K}$ where $K_1 \leq K_2$ if there exists a homomorphism from $K_1$ to $K_2$. By Lemma \ref{lem:core homomorphism poset}, this relation is a poset. Hence, it has a minimal element $K_0$. Let $H_0 \in \mathcal{H}$ be such that
$K_0 = \mathrm{core}(H_0)$. The choice of $K_0$ gives the following:
\begin{fact}\label{fact:minimal core}
    For every $H' \in \mathcal{H}$, if $H'$ is homomorphic to $K_0$ then $\mathrm{core}(H') \cong K_0$.
\end{fact}
\begin{proof}
    Restricting a homomorphism $H'\to K_0$ to the subgraph
    $K':=\mathrm{core}(H')$ gives a homomorphism $K'\to K_0$.  By the
    minimality of $K_0$ in the homomorphism order on $\mathcal K$, we have
    $K'\cong K_0$.
\end{proof}

With the above preparations, we can now proceed to describe the construction which gives
$\delta_{\mathrm{VC}}(H) > \frac{r-3}{r-2}$.
Fix any $d \geq 1$, and let $F$ be an $m \times m$ bipartite graph with $\VCdim(F) \geq d$, where $m$ depends only on $d$. We identify both sides of $F$ with $[m]$.
Fix $N \gg m$, to be chosen later.
Suppose that $V(K_0) = \{1,\dots,k\}$.
We have $\chi(K_0) \geq 3$ (since
$\chi(H_0) \geq 3$ and $K_0$ is the core of $H_0$).
Note that $\delta(K_0) \geq 2$ (since otherwise $K_0$ would be homomorphic to a proper subgraph of itself).
Every connected graph on $\geq 2$ vertices has a vertex whose deletion leaves the graph connected. Without loss of generality, assume that the connected component $C$ of $K_0$ containing the vertex $1$ stays connected after removing $1$. Suppose also that $12 \in E(K_0)$.

Fix $t \gg m$, to be chosen (implicitly) later.
Take disjoint sets $U_{i,\ell}$ for $i \in [k]$ and $\ell \in [m]$, where $|U_{i,\ell}| = t$.
Put $U_i = \bigcup_{\ell=1}^m U_{i,\ell}$ for $i \in [k]$. We start by defining a graph $G_0$ on $U_1 \cup \dots \cup U_k$, as follows:
\begin{itemize}
    \item For each $ij \in E(K_0)$ with $i,j \neq 1$ and each $\ell \in [m]$, the bipartite graph $(U_{i,\ell},U_{j,\ell})$ is complete. (Thus, the bipartite graph $(U_i,U_j)$ is a disjoint union of copies of $K_{t,t}$.)
    \item For each $(h,\ell) \in [m]^2$, the bipartite graph $(U_{1,h},U_{2,\ell})$ is complete if $h\ell \in E(F)$ and empty otherwise.
    Thus, the bipartite graph $(U_1,U_2)$ is the $t$-blowup of $F$.
    \item For each $1j \in E(K_0)$ with $j \neq 2$, and for each $(h,\ell) \in [m]^2$,
    the bipartite graph $(U_{1,h},U_{j,\ell})$ is complete if $h\ell \notin E(F)$ and empty otherwise.
    Thus, the bipartite graph $(U_1,U_j)$ is the $t$-blowup of $\overline{F}$, the bipartite complement of $F$.
\end{itemize}
This completes the definition of $G_0$. Clearly, $G_0$ is homomorphic to $K_0$ via the homomorphism mapping $U_i$ to $i$ for all $i \in [k]$.
The following two claims capture the key properties of $G_0$.
\begin{claim}\label{claim:K_0-free}
    $G_0$ is $K_0$-free.
\end{claim}
\begin{poc}
    Suppose by contradiction that there is a copy of $K_0$ in $G_0$. By Lemma \ref{lem:core homomorphism}, for every $i\in [k]$, the
    $K_0$-copy has exactly one vertex in $U_i$; denote this vertex by $u_i$. Thus,
    $u_iu_j \in E(G_0)$ if and only if
    $ij \in E(K_0)$. For $i \in [k]$, let
    $\ell_i \in [m]$ such that $u_i \in U_{i,\ell_i}$. 
    Without loss of generality, suppose that $3$ is a neighbor of $1$ in $K_0$ (we know that $\delta(K_0) \geq 2$).
    Consider the connected component $C$ of $K_0$ containing $1$. By assumption, removing 1 from $C$ leaves it connected. Now, tracing the edges of $C \setminus \{1\}$, we see that $\ell_i = \ell_j$ for all $i,j \in C \setminus \{1\}$ (this follows from the first item in the definition of $G_0$). In particular, $\ell_2 = \ell_3$. However, the fact that $u_1u_2 \in E(G_0)$ means that $\ell_1\ell_2 \in E(F)$, and the fact that $u_1u_3 \in E(G_0)$ means that $\ell_1\ell_3 \notin E(F)$ (here we use the second and third items in the definition of $G_0$, respectively). This contradicts $\ell_2 = \ell_3$.
\end{poc}
\begin{claim}\label{claim:H_0-copy}
    Every supergraph $G$ of $G_0$ having $\VCdim(G) < d$ contains a copy of $H_0$.
\end{claim}
\begin{poc}
    Recall that $G_0[U_1,U_2]$ (i.e., the bipartite graph between $U_1$ and $U_2$ in $G_0$) is the $t$-blowup of $F$. Hence, $G_0[U_1,U_2]$ contains at least $t^{2m}$ bi-induced copies of $F$. Since $G$ is assumed to have VC-dimension less than $d$, none of these bi-induced copies of $F$ can appear in $G$. Note also that each pair in
    $U_1 \times U_2$ can participate in exactly
    $t^{2m-2}$ of these bi-induced copies of $F$. It follows that
    $e_G(U_1,U_2) - e_{G_0}(U_1,U_2) \geq t^2$, i.e., at least $t^2$ edges must be added to $G_0[U_1,U_2]$ to eliminate all bi-induced copies of $F$.
    By the pigeonhole principle, there exist
    $h,\ell \in [m]$ such that
    $e_G(U_{1,h},U_{2,\ell}) -
    e_{G_0}(U_{1,h},U_{2,\ell}) \geq t^2/m^2$.
    Now consider the bipartite graph $B$ with parts $U_{1,h},U_{2,\ell}$ and whose edges are all non-edges of $G_0$ (between $U_{1,h},U_{2,\ell}$) which are edges in $G$.
    Thus, both parts of $B$ have size $t$, and
    $e(B) \geq t^2/m^2$. Therefore, if $t$ is large enough compared to $m$, then the
    K\H{o}v\'ari-S\'os-Tur\'an theorem \cite{KST} gives a complete bipartite graph in $B$ with parts of size $v := |V(H_0)|$. Denote the parts of this $K_{v,v}$-copy by $S_1 \subseteq U_{1,h}$ and
    $S_2 \subseteq U_{2,\ell}$.
    Put also $S_i := U_{i,\ell}$ for $3 \leq i \leq k$.
    We claim that $(S_i,S_j)$ is complete in $G$ for every
$ij\in E(K_0)$.  Since $H_0\to K_0$ and every $S_i$ has size at
least $v(H_0)$, this yields a copy of $H_0$ in $G$. If $i,j \neq 1$ then this holds by the definition of $G_0$, and for $ij = 12$ it holds by the choice of $S_1,S_2$. Consider now the case $i=1, j \neq 2$. Note that the bipartite graph $G_0[U_{1,h},U_{2,\ell}]$ is empty; indeed, this bipartite graph is either complete or empty, and it cannot be complete because otherwise $B$ would have no edges. It follows that
    $h\ell \notin E(F)$. In turn, this means that $(S_1,S_j) = (U_{1,h},U_{j,\ell})$ is complete in $G_0$, as required.
\end{poc}

Next, we guarantee large minimum degree. To this end, fix $N \gg t$, and add vertex-sets $X_1,\dots,X_k$ and $Y_1,\dots,Y_{r-3}$ with $|X_i| = N$ and
$|Y_i| = (k-1)N$. Put $U := \bigcup_{i=1}^k U_i$,
$X := \bigcup_{i=1}^k X_i$ and
$Y := \bigcup_{i=1}^{r-3}Y_i$. To define the new edges added to the graph, we need Lemma \ref{lem:orientation}:
Let $D$ be an orientation of a spanning subgraph of $K_0$ such that $d^+_D(i) = 1$ for all $i \in [k]$ (note that $K_0$ satisfies the condition of Lemma \ref{lem:orientation} because $\delta(K_0) \geq 2$).
The underlying graph of $D$ has no isolated vertices.  In each of its
connected components choose a spanning tree, and let $S$ be the resulting
oriented spanning forest.  Thus $S\subseteq D$ and $S$ has no isolated
vertices.
We now add the following complete bipartite graphs:
\begin{itemize}
    \item $(U \cup X,Y)$;
    \item $(Y_i,Y_j)$ for all $1 \leq i < j \leq r-3$;
    \item $(U_i,X_j)$ for every directed edge $(i,j) \in E(D)$;
    \item $(X_i,X_j)$ for every edge
    $(i,j) \in E(S)$.
\end{itemize}
Denote the resulting graph by $G'$. Then
$$
n := |V(G')| = |U| + (k + (r-3)(k-1))N =
((r-2)(k-1) + 1 + o(1))N.
$$
Note also that $G'[U \cup X]$ is homomorphic to $K_0$ via the homomorphism mapping $U_i \cup X_i$ to $i$ for each $i \in [k]$; this is because $S$ and $D$ are both subgraphs of $K_0$.
\begin{claim}
    $\delta(G') \geq ((r-3)(k-1)+1)N =
    \left( \frac{(r-3)(k-1)+1}{(r-2)(k-1)+1} - o(1) \right) n$.
\end{claim}
\begin{poc}
    We need to show that
    $d_{G'}(v) \geq ((r-3)(k-1)+1)N$ for all $v \in V(G')$. Suppose first that $v \in U_i$ for some $i \in [k]$.
    Since $d^+_D(i) = 1$, there exists $j \in [k]$ with $(i,j) \in E(D)$.
    By the definition of $G'$, the vertex $v \in U_i$ is adjacent to all vertices in $X_j$, and also to all vertices in $Y$. Hence, $d_{G'}(v) \geq ((r-3)(k-1)+1)N$, as required.

    Suppose now that $v \in X_i$ for some
    $i \in [k]$. Let $j \in [k]$ such that $(i,j)$ or $(j,i)$ belongs to $S$ (such a vertex $j$ exists due to the choice of $S$). Then
    $d_{G'}(v) \geq |X_j| + |Y| = ((r-3)(k-1)+1)N$.

    Finally, suppose that $v \in Y_i$ for some
    $i \in [r-3]$. Then
    $$
    d_{G'}(v) \geq |X| + \sum_{j \neq i}|Y_j| =
    kN + (r-4)(k-1)N = ((r-3)(k-1)+1)N,
    $$
    as required.
\end{poc}
Since $k \leq v(H)$, we get that
$\delta(G') \geq
\left(
\frac{(r-3)(v(H)-1)+1}{(r-2)(v(H)-1)+1} - o(1) \right) n$.
To complete the proof of the theorem, we now prove the following:
\begin{claim}
    $G'$ is $H$-free.
\end{claim}
\begin{poc}
    Suppose by contradiction that there is an embedding $\varphi : H \rightarrow G'$.
    Let $B := \varphi^{-1}(U \cup X)$, and for each $i \in [r-3]$ let $A_i := \varphi^{-1}(Y_i)$. Then $A_1,\dots,A_{r-3}$ are independent. Hence,
    $H' := H[B] \in \mathcal{H}$.
    Let $K'$ be the core of $H'$.
    As mentioned above, $G'[U \cup X]$ is homomorphic to $K_0$ via the homomorphism $U_i \cup X_i \mapsto i$. Therefore, $H'$ (and hence also $K'$) is homomorphic to $K_0$. By Fact \ref{fact:minimal core}, we conclude that
    $K' \cong K_0$. By Lemma \ref{lem:core homomorphism}, for every $i \in [k]$, the set $\varphi(K')$ contains exactly one vertex from $U_i \cup X_i$; denote this vertex by $v_i$.
    Then $v_iv_j \in E(G')$ if and only if $ij \in E(K_0)$. 

    We claim that $v_i \in U_i$ for all $i \in [k]$ (i.e., having $v_i \in X_i$ is impossible).
    So suppose by contradiction that the set
    $I := \{i \in [k] : v_i \in X_i\}$ is non-empty. Let $i \in I$ such that $d^+_{S[I]}(i)=0$; such an $i$ exists because $S[I]$ is an oriented forest.
    Observe that for each $j \in [k] \setminus \{i\}$, if $v_iv_j \in E(G')$ then
    $j \in N^-_D(i)$. First, if $v_j \in U_j$ then this indeed holds by the definition of the graph $G'$. Suppose now that $v_j \in X_j$. Then $v_iv_j \in E(G')$ implies that $(i,j) \in E(S)$ or $(j,i) \in E(S)$. However, $j \in I$ so $(i,j) \in E(S)$ is impossible due to the choice of $i$. Hence, $(j,i) \in E(S)$, meaning that $j \in N^-_D(i)$, as claimed.
    It now follows that
    \[
    d_{\varphi(K')}(v_i)
    \le d^-_D(i)
    \le d_{K_0}(i)-d^+_D(i)
    = d_{K_0}(i)-1.
    \]
    However, this contradicts the fact that $\varphi(K') = \{v_1,\dots,v_k\}$ forms a copy of $K_0$ with $v_i$ playing the role of $i$. Thus, $I = \emptyset$, proving our claim that $v_i \in U_i$ for all $i \in [k]$.

    Finally, we see that $v_1,\dots,v_k$ form a copy of $K_0$ in $G_0$. But this contradicts Claim \ref{claim:K_0-free}.
\end{poc}

Let $G$ be a maximal $H$-free supergraph of $G'$.  We claim that
$\VCdim(G)\ge d$.  Otherwise $\VCdim(G[U])<d$, and $G[U]$ is a
supergraph of $G_0$.  Claim~\ref{claim:H_0-copy} therefore gives a copy of
$H_0$ in $G[U]$.  Since $H_0\in\mathcal H$, there is a partition
$V(H)=A_1\cup\cdots\cup A_{r-3}\cup B$ with the $A_i$ independent and
$H[B]\cong H_0$.  Choosing $|A_i|$ vertices from $Y_i$ for each $i$ and
using the complete joins already present in $G'$ extends this copy of $H_0$
to a copy of $H$ in $G$, a contradiction.  Hence $\VCdim(G)\ge d$.
\end{proof}

\begin{remark}
    It is likely that the bound in Theorem \ref{thm:strict lowbd for delta_VC} can be improved for various graphs $H$, by modifying the way the sets $X_1,\dots,X_k$ are connected among themselves and to the sets $U_1,\dots,U_k$.
\end{remark}

\section{Proof of \Cref{thm:value for graphs of delta_chi(VC)}}\label{sec:chromatic VC}

In this section, we prove \Cref{thm:value for graphs of delta_chi(VC)}.
For the lower bounds in Theorem \ref{thm:value for graphs of delta_chi(VC)}, we now recall a standard construction used in the determination of the chromatic threshold by Allen et al.~\cite{ALLEN2013261}.

\begin{construction}[$G_{t,h,C}$]\label{construction:VC chromatic lower bound}
    By a famous result of Erd\H{o}s \cite{1959Erdos}, there exists a graph $J$ with $\operatorname{girth}(J)>\max\{h,4\}$ and
    $\chi(J)>C$. Put $N:=|V(J)|$.
    Let $T_t(N)$ denote the complete $t$-partite graph whose parts all have size
    $N$, where $T_0(N)$ is the empty graph.
    Now define
    $G_{t,h,C}:= J\vee T_t(N)$.
\end{construction}

\begin{claim}\label{claim:VC chromatic lower bound}
    Let $G := G_{t,h,C}$. Then
    \begin{enumerate}
        \item
        $\delta(G)\ge \frac{t}{t+1}|V(G)|$ and $\chi(G) > C$.
        \item $\VCdim(G) \leq 3$.
        \item Let $H$ be a graph with at most $h$ vertices. If $\chi(H) \geq t+3$, or $\chi(H) = t+2$ and $\mathcal{M}(H)$ contains no forest, then $G$ is $H$-free.
    \end{enumerate}
\end{claim}
\begin{proof}
    We have
    $\chi(G) \geq \chi(J) > C$ and
    $|V(G)| = (t+1)N$. It is also easy to see that $\delta(G) \geq tN$, proving Item 1.
    Next, we prove Item 2. The graph $J$ is $C_4$-free, so \Cref{fact:VC-implies-C4} gives
    $\VCdim(J)\le 2$.
    Also, $\VCdim(T_t(N)) \leq 2$ by Fact \ref{fact:VC complete multipartite}. Hence, $\VCdim(G) \leq 3$ by \Cref{fact:VC-of-join}.

    Finally, we prove Item 3. Let $H$ be a graph on at most $h$ vertices, and suppose that $G$ contains a copy $H'$ of $H$. Since
    $\chi(T_t(N)) = t$, the subgraph of $H'$ induced by $V(H') \cap V(J)$ has chromatic number at least $\chi(H) - t$. On the other hand, this subgraph has at most $h$ vertices, and every subgraph of $J$ on $h$ vertices is a forest (as $\mathrm{girth}(J) > h$). It follows that $\chi(H) \leq t+2$, and if $\chi(H) = t+2$ then $\mathcal{M}(H)$ contains a forest. This proves (the contrapositive of) Item 3.
\end{proof}

\begin{comment}
    For the upper bound, we need the following partitioning tool. It follows
from Haussler's packing bound for set systems of bounded
VC-dimension~\cite{1995PackingLemma}; see also
\cite{luczak_coloring_2010} and the formulation in
\cite{liu_beyond_2024}.

\begin{lemma}[Partition lemma {\cite{liu_beyond_2024}}]
    \label{lemma:Partition}
    Let $d$ be a positive integer and let $G$ be an $n$-vertex graph with
    VC-dimension at most $d$. For every $1\le a\le n$, there is a
    partition $V(G)=V_1\sqcup\cdots\sqcup V_m$ with
    $m\le e(d+1)(2e)^d\left(\frac{n}{a}\right)^d$
    such that $|N_G(u)\triangle N_G(v)|\le 2a$ for every $i\in[m]$ and
    every $u,v\in V_i$.
\end{lemma}
\end{comment}

\noindent
Finally, we record a standard fact about forests.

\begin{fact}\label{fact:chromatic number large forest}
    Every graph of chromatic number at least $k$ contains every forest on
    at most $k$ vertices.
\end{fact}

\begin{proof}
    Let $G'$ be a subgraph of $G$ that is minimal subject to
    $\chi(G')\ge k$. Then $\delta(G')\ge k-1$. Every forest on at most
    $k$ vertices can now be embedded greedily into $G'$, one component at
    a time and following a rooted order in each component.
\end{proof}

\begin{proof}[Proof of \Cref{thm:value for graphs of delta_chi(VC)}]
    Let $H$ be a graph and put $r := \chi(H)$.
    If $r=2$, then the Erd\H{o}s--Stone theorem gives
    $\operatorname{ex}(n,H)=o(n^2)$. Hence, for every $\eps>0$ and all
    sufficiently large $n$, every graph with minimum degree at least
    $\eps n$ contains $H$. Therefore
    $\delta_{\chi}^{\mathrm{VC}}(H)=0$.

    Assume from now on that $r\ge 3$. Put $h:=|V(H)|$. We first prove the lower bounds in the theorem.
    Fix an arbitrary $C > 0$.
    First let $G := G_{r-3,h,C}$ as in Construction \ref{construction:VC chromatic lower bound}.
    By Claim \ref{claim:VC chromatic lower bound}, we have $\delta(G) \geq \frac{r-3}{r-2}|V(G)|$, $\chi(G) > C$, $\VCdim(G) \leq 3$, and $G$ is $H$-free. This shows that $\delta_{\chi}^{\mathrm{VC}}(H)\ge\frac{r-3}{r-2}$.
    Suppose now that $\mathcal M(H)$ contains no forest, and take
    $G := G_{r-2,h,C}$. Then by Claim \ref{claim:VC chromatic lower bound}, we have
    $\delta(G) \geq \frac{r-2}{r-1}|V(G)|$, $\chi(G) > C$, $\VCdim(G) \leq 3$, and $G$ is $H$-free. This proves that
    $\delta_{\chi}^{\mathrm{VC}}(H)\ge\frac{r-2}{r-1}$ whenever
    $\mathcal M(H)$ contains no forest.

    We now move on to the upper bounds in \Cref{thm:value for graphs of delta_chi(VC)}.
    First, the Erd\H{o}s--Stone theorem gives
    $\delta_{\chi}^{\mathrm{VC}}(H)\le\frac{r-2}{r-1}$. Hence, it remains only
    to prove that
    $\delta_{\chi}^{\mathrm{VC}}(H)\le\frac{r-3}{r-2}$ when
    $\mathcal M(H)$ contains a forest.
    Let $F\in\mathcal M(H)$ be a forest. Fix any $0<\eps\le 1$ and $d \geq 1$. We shall prove that there exists
    $n_0=n_0(H,\eps,d)$ such that every $H$-free graph $G$ on
    $n\ge n_0$ vertices satisfying $\VCdim(G)\le d$ and
    $\delta(G)\ge
    \big(\frac{r-3}{r-2}+\eps\big)n$
    has chromatic number bounded in terms of $H$, $\eps$, and $d$ only.

    Apply \Cref{lemma:Partition} with
    $a:=\frac{\eps n}{10h}$. For all sufficiently large $n$, this
    gives a partition $V(G)=V_1\sqcup\cdots\sqcup V_m$, where $m$ is
    bounded in terms of $H$, $\eps$, and $d$, such that
    $|N_G(u)\triangle N_G(v)|\le 2a = \frac{\eps n}{5h}$ whenever
    $u,v$ belong to the same part.
    In particular, if $W$ is a non-empty subset of some $V_i$ with
    $|W|\le h$, then, fixing any $u\in W$, we obtain
    \begin{align}\label{ineq: large common neighbor}
        |N_G(W)|
        =\left|\bigcap_{v\in W}N_G(v)\right|\ge |N_G(u)|
        -\sum_{v\in W}|N_G(u)\setminus N_G(v)|\ge \delta(G)-2a|W|
        \ge
        \left(\frac{r-3}{r-2}+\frac{\eps}{2}\right)n.
    \end{align}

    \begin{claim}\label{claim: large part no forest}
        For each $i \in [m]$,
        $G[V_i]$ is $F$-free.
    \end{claim}

    \begin{poc}
        Suppose for a contradiction that $G[V_i]$ contains a copy $F_1$
        of $F$, and put $W_1:=V(F_1)$. We construct sets
        $W_2,\ldots,W_{r-1}$ inductively. Each $W_j$ for $2 \leq j \leq r-1$ will have size $h$
        and will be contained in one part of the partition, and every two
        distinct sets among $W_1,\ldots,W_{r-1}$ will be completely
        joined.

        Suppose that $W_1,\ldots,W_j$ have been chosen for some
        $1\le j\le r-2$. By \eqref{ineq: large common neighbor},
        \[
            \left|\bigcap_{\ell=1}^j N_G(W_\ell)\right|
            \ge
            \left(
                1-\frac{j}{r-2}+\frac{j\eps}{2}
            \right)n.
        \]
        Since $j\le r-2$, the right-hand side is linear in $n$. The
        number $m$ of partition classes is independent of $n$, so for
        all sufficiently large $n$, some class $V_t$ contains at least
        $h$ vertices of this common neighborhood. Let $W_{j+1}$ be any
        $h$ such vertices. Then $W_{j+1}$ is completely joined to every
        one of $W_1,\ldots,W_j$, and it lies in a single partition class,
        so the induction may continue.

        We eventually obtain $W_1,\ldots,W_{r-1}$ such that every two of
        these sets are completely joined and $G[W_1]$ contains a copy of
        $F$. Since $F\in\mathcal M(H)$, the definition of the
        decomposition family implies that
        $G[W_1\cup\cdots\cup W_{r-1}]$ contains a copy of $H$, a
        contradiction.
    \end{poc}

    We can now complete the proof:
    for each $i \in [m]$,
    \Cref{claim: large part no forest} and
    \Cref{fact:chromatic number large forest} give
    $\chi(G[V_i])<h$. Hence $\chi(G)<mh$, which is bounded in terms of
    $H$, $\eps$, and $d$ only. This proves
    $\delta_{\chi}^{\mathrm{VC}}(H)\le\frac{r-3}{r-2}$ whenever
    $\mathcal M(H)$ contains a forest, and completes the proof.
\end{proof}

\section{Concluding remarks}\label{sec:remark}
The most immediate problem for homomorphism thresholds is
Conjecture~\ref{conj:delta_hom K(1,s,t)}, which predicts the exact value of
$\delta_{\mathrm{hom}}(K_{1,s,t})$ for all $s\le t$.  The case
$K_{1,6,8}$ is an important unresolved instance: it seems to require
coloring information beyond the present Brooks-type argument and is connected
to the Borodin--Kostochka problem.

A natural next step is to identify a broader structural criterion for
$3$-chromatic graphs with $\delta_{\mathrm{hom}}(H)=1/3$.  The graph
$P_4^+$, obtained by adding a universal vertex to a four-vertex path, is a
basic test case, and we expect $\delta_{\mathrm{hom}}(P_4^+)=1/3$. Another interesting direction, in view of \Cref{thm:strict lowbd for delta_VC}, is whether the
homomorphism threshold is also always positive.

\begin{problem}
    Does every graph $H$ with $\chi(H)=r\ge 3$ satisfy
    $\delta_{\mathrm{hom}}(H)>0$?
\end{problem}

For VC-dimension thresholds, the principal comparison problem is
\Cref{conj:chromatic VC}.  Our results settle it whenever the chromatic
threshold is one of the two lower values, leaving precisely the top regime
$\delta_\chi(H)=\frac{r-2}{r-1}.$
Understanding this case appears to require a structural link between maximal
$H$-free graphs of unbounded VC-dimension and the decomposition-family
obstruction governing the chromatic threshold.  Even for $3$-chromatic
forbidden graphs, the following classification problem remains open.
\begin{problem}
Characterize the graphs $H$ for which
$\delta_{\mathrm{VC}}(H)=\frac12$.
\end{problem}

Theorem~\ref{thm:complete results on tripartite graphs on VC} in
particular shows that $1/2$ is an accumulation point of
$\delta_{\mathrm{VC}}(H)$. The ideas here can be extended
to show that $(r-2)/(r-1)$ is an accumulation point for every $r\ge3$;
we omit the details.

\paragraph*{Acknowledgements.}
The second author would like to thank the 2025 IBS ECOPRO Summer School where this project was initiated. While finalizing this manuscript, we learned that Jinze Hu, Qinghai Liu, and Liping Zhang independently obtained \Cref{thm:value for graphs of delta_chi(VC)}. We thank them for informing us of their work.

\bibliographystyle{abbrv}
\bibliography{HomandVC}
\end{document}